\documentclass[11pt,a4paper]{amsart}

\usepackage{amsmath}
\usepackage{amsfonts}
\usepackage{graphicx}
\usepackage{framed}
\usepackage{amssymb,cancel}
\usepackage{xcolor}
\usepackage{esint}
\usepackage{comment}

\usepackage{etex}
\usepackage{latexsym}
\usepackage[english]{babel}
\usepackage[utf8x]{inputenc}

\def \N {\mathbb{N}}
\def \R {\mathbb{R}}
\def \XX {\mathcal{X}}

\def \de {\partial}

\def \LL {\mathcal{L}}

\def \Oo {\mathcal{O}}

\usepackage{amsthm}
\theoremstyle{definition}
\newtheorem{definition}{Definition}[section]

\newtheorem{remark}[definition]{Remark}

\theoremstyle{plain}
\newtheorem{theorem}[definition]{Theorem}
\newtheorem{proposition}[definition]{Proposition}
\newtheorem{lemma}[definition]{Lemma}
\newtheorem{corollary}[definition]{Corollary}

\numberwithin{equation}{section}
\usepackage[colorlinks=true,urlcolor=blue,
citecolor=red,linkcolor=blue,linktocpage,pdfpagelabels,
bookmarksnumbered,bookmarksopen]{hyperref}
\usepackage[hyperpageref]{backref}

\makeatletter
\@namedef{subjclassname@2020}{\textup{2020} Mathematics Subject Classification}
\makeatother

\begin{document}

 \title[Local-nonlocal singular problems with absorption]
 {Regularizing effects of absorption terms in local-nonlocal mild singular problems}

 \author[S.\,Biagi]{Stefano Biagi}
 \author[E.M.\,Merlino]{Enzo Maria Merlino}
 \author[E.\,Vecchi]{Eugenio Vecchi}
 
 \address[S.\,Biagi]{Dipartimento di Matematica
 \newline\indent Politecnico di Milano \newline\indent
 Via Bonardi 9, 20133 Milano, Italy}
 \email{stefano.biagi@polimi.it}

 \address[E.M.\,Merlino]{Dipartimento di Matematica
 \newline\indent Università di Bologna \newline\indent
 Piazza di Porta San Donato 5, 40126 Bologna, Italy}
 \email{enzomaria.merlino2@unibo.it} 
 
 \address[E.\,Vecchi]{Dipartimento di Matematica
 \newline\indent Università di Bologna \newline\indent
 Piazza di Porta San Donato 5, 40126 Bologna, Italy}
 \email{eugenio.vecchi2@unibo.it}

\subjclass[2020]{35J75, 35M12, 35J61, 35B51}

\keywords{Local-nonlocal operators, singular PDEs, comparison principles}

\thanks{The Authors are
member of the {\em Gruppo Nazionale per
l'Analisi Ma\-te\-ma\-ti\-ca, la Probabilit\`a e le loro Applicazioni}
(GNAMPA) of the {\em Istituto Nazionale di Alta Matematica} (INdAM) and are
partially 
supported by the PRIN 2022 project 2022R537CS \emph{$NO^3$ - Nodal Optimization, NOnlinear elliptic equations, NOnlocal geometric problems, with a focus on regularity}, founded by the European Union - Next Generation EU.  S.B. and E.V. are also partially supported by the Indam-GNAMPA project CUP E5324001950001 - {\em Problemi singolari e degeneri: esistenza, unicità e analisi delle proprietà qualitative delle soluzioni}. E.M.M. is also partially supported by the Indam-GNAMPA project CUP E5324001950001 - {\em Ottimizzazione spettrale, geometrica e funzionale}. \\
The Authors thank F. Oliva for many useful and fruitful discussions on the topics of the paper}

 \begin{abstract}
In this paper we prove existence and uniqueness of energy solutionns for singular problems with absorption driven by local-nonlocal operators. Moreover, 
we establish a comparison principle \`a la Talenti, leading to a gain of summability result for the solutions of these problems.
 \end{abstract}
 
 \maketitle 
 
\section{Introduction}\label{sec.Intro}
Local-nonlocal operators have been studied in connection with the validity of maximum principles since the '60s, see e.g. \cite{BCP, Cancelier}. In the last twenty years there has been a restored interest thanks to the their natural appearance as stochastic processes combining a Brownian motion and a Lévy-type process, see e.g. \cite{CKSV} and the references therein. In this paper we are interested in working with the simplest possible realization of such operators, namely
\begin{equation*}
	\mathcal{L}:=	-\Delta + (-\Delta)^s, \quad s \in (0,1),
\end{equation*}
\noindent where $(-\Delta)^s$ denotes the fractional Laplacian. We note that the above operator falls in the general framework considered in \cite{BFV} thanks to the Green function estimates proved in \cite{CKSV}. Despite being linear, the operator $\mathcal{L}$ exhibits a natural lack of homogeneity which may cause significant differences w.r.t. to $-\Delta$ or $(-\Delta)^s$: for example, this is the case in critical problems, see e.g. \cite{BDVV5, BiagiVecchi, BiagiVecchi2}.
In the last five years, a more analytical approach towards regularity properties of $\mathcal{L}$ (and its quasilinear variants) has been developed, see e.g. \cite{BDVV, AC, GarainKinnunen, GarainLindgren, DeFMin, SVWZ, BMS, BySo, DiFaZh, AntoCozzi}. We stress in particular that in \cite{SVWZ2} it has been recently showed that natural $C^{2,\alpha}$ Schauder estimates up to the boundary may fail. We finally mention that operators like $\mathcal{L}$ may appear in biomathematical models, see e.g., \cite{DPLV2, DV} and the references therein.
\medskip

We are now ready to describe the problem we are interested in.
Let $\Omega \subset \mathbb{R}^{n}$ be an open and bounded set with smooth enough boundary $\partial \Omega$.
We consider the following local-nonlocal problem:
\begin{equation}\label{eq:Problem}
	\left\{ \begin{array}{rl}
		\mathcal{L} u + g(u) = h(u)f & \textrm{ in } \Omega,\\
		u>0 & \textrm{ in } \Omega,\\
		u=0 & \textrm{ in } \mathbb{R}^{n}\setminus \Omega.
	\end{array}\right.
\end{equation}
To be more precise, we are interested in proving existence of finite energy solution to \eqref{eq:Problem} according to the summability of $f$. We postpone for the moment the precise assumptions on the semilinear terms $g$ and $h$ appearing in \eqref{eq:Problem}.\\
\noindent The study of regularity properties of solutions to elliptic PDEs depending on the summability of the source datum is a pretty classical subject that dates back to results by Calderon and Zygmund \cite{CalZyg}  and Stampacchia \cite{Stampacchia}. The first classical instance is provided by the classical Dirichlet problem 
\begin{equation}\label{eq:ClassicoDirichlet}
	\left\{\begin{array}{rl}
		-\Delta u = f(x) & \textrm{in } \Omega,\\
		u=0 & \textrm{on } \partial \Omega,
	\end{array}\right.
\end{equation}
\noindent with $f \in L^{m}(\Omega)$ for some $m\geq 1$. In particular, see also \cite{BoGa}, the following holds:
\begin{itemize}
	\item if $m>\tfrac{n}{2}$, then $u \in H^{1}_{0}(\Omega)\cap L^{\infty}(\Omega)$;
	\item if $\tfrac{2n}{n+2}\leq m \leq \tfrac{n}{2}$, then $u \in H^{1}_{0}(\Omega)\cap L^{m^{\ast \ast}}(\Omega)$ with $m^{\ast \ast}:=\tfrac{nm}{n-2m}$;
	\item if $1<m<\tfrac{2n}{n+2}$, then $u \in W^{1,m^{\ast}}_{0}(\Omega)$ with $m^{\ast}:=\tfrac{n m}{n-m}$;
	\item if $m=1$, then $u\in W^{1,t}_{0}(\Omega)$ for every $t<\tfrac{n}{n-1}$.
\end{itemize}
We stress that the notion of solution has to be modified accordingly to the case considered. We refer to \cite{OP2} for a more general singular term and we also refer to \cite{LPPS} for the appropriate analogous results in the fractional case, and to \cite{ArRa} for the local-nonlocal case. \\
\noindent In \cite{BO}, Boccardo and Orsina considered the following singular problem 
\begin{equation}\label{eq:BoccardoOrsina}
	\left\{\begin{array}{rl}
		-\Delta u = \dfrac{f(x)}{u^{\gamma}} & \textrm{in } \Omega,\\
		u=0 & \textrm{on } \partial \Omega,
	\end{array}\right.
\end{equation}
\noindent with $0\leq f \in L^{m}(\Omega)$ for some $m\geq 1$ and with $\gamma >0$. They proved that the solution has an improvement in terms of Sobolev regularity w.r.t. the nonsigular case $\gamma=0$. Precisely, they proved that
\begin{itemize}
	\item if $\gamma =1$ and $m=1$, then $u \in H^{1}_{0}(\Omega)$;
	\item if $\gamma <1$ and $m=\tfrac{2n}{n(1+\gamma)+2(1-\gamma)}$, then $u \in H^{1}_{0}(\Omega)$;
	\item if $\gamma <1$ and $m<\tfrac{2n}{n(1+\gamma)+2(1-\gamma)}$, then $u \in W^{1,t}_{0}(\Omega)$ with $t = \tfrac{nm(1+\gamma)}{n-m(1-\gamma)}$.
\end{itemize}
We omit here the details concerning the case $\gamma>1$ because this is out of our interest, since this leads to solutions with possibly infinite energy. Once again, similar results have already been proved in \cite{BdBMP, YoMa} for the fractional case and in \cite{ArRa} in the local-nonlocal case. We notice here that other singular perturbations, e.g. with Hardy-type potentials, lead to quite different results, see \cite{BOP} and \cite{BEMV}.\\
A similar regularizing effect in terms of Sobolev regularity can also be obtained perturbing \eqref{eq:ClassicoDirichlet} with a power-type absorption term as in \cite{BoGaVa, Cirmi}. Just to give a glance at the situation occurring in presence of an absorption term, we recall that if $f\in L^{1}(\Omega)$ and $q>\tfrac{n}{n-2}$, there exists a solution $u$ to 
\begin{equation}\label{eq:Boccardogallouet}
	\left\{\begin{array}{rl}
		-\Delta u + u^q = f(x) & \textrm{in } \Omega,\\
		u=0 & \textrm{on } \partial \Omega,
	\end{array}\right.
\end{equation}
\noindent such that $u \in W^{1,t}_{0}(\Omega)$ for every $t<\tfrac{2q}{q+1}$, yielding an improvement w.r.t. the solution of \eqref{eq:ClassicoDirichlet}. It is worth to stress that we are not aware of any other contribution in this direction neither in the purely nonlocal case, nor in the local-nonlocal one.\\
It is therefore natural to wonder whether the combination of both the regularizing effects, i.e. absorption and singular term, may provide an even stronger regularization. This question has been addressed in \cite{DeCOl} and then extended to more general absorption and singular terms in the case of the $p$-Laplacian in \cite{Oliva}. We notice that Oliva in \cite{Oliva} considered also the case $\gamma >1$ which leads to solutions with locally finite energy. We will not deal with this case in this paper.
Our interest is therefore only to find energy solutions to \eqref{eq:Problem}. 
\medskip

Let us now move back to problem \eqref{eq:Problem}, and let us begin
by fixing all the \emph{standing assumptions} on the \emph{absorption term $g$}, on the \emph{singular term $h$} and on $f$.
\medskip

\noindent {\bf Standing assumptions:}
Throughout what follows, we make
the following \emph{structural assumptions} on the functions $g,h$ and $f$
involved in problem \eqref{eq:Problem}.
\begin{itemize}
\item[(H)$_h$] $h:[0,+\infty)\to (0,+\infty]$ is continuous, possibly singular, finite outside the origin and such that $h(0)\neq 0$. Moreover,
\begin{equation}\label{eq:assh1}\tag{h1}
\exists\,\,\gamma \in [0,1],\,\underline{C},\,\underline{s}>0: h(s) \leq \dfrac{\underline{C}}{s^{\gamma}}, \textrm{ for every } s \leq \underline{s},
\end{equation}
\noindent and
\begin{equation}\label{eq:assh2}\tag{h2}
\exists\,\,\theta \geq 0,\,\overline{C},\,\overline{s}>\underline{s}: h(s) \leq \dfrac{\overline{C}}{s^{\theta}}, \textrm{ for every } s \geq \overline{s}.
\end{equation}
	\item[(H)$_f$]  $f:\Omega\to[0,+\infty)$ is a measurable function such that \begin{equation} \label{eq:assumptionMeasurefzero}
  |\{x\in\Omega:\,f(x) = 0\}| = 0 \quad \textrm{for } 1/2<s<1, 
 \end{equation}
(where $|A|$ denotes the Lebesgue measure of a measurable set $A\subseteq\R^n$);
moreover, we suppose that
$$\mathrm{i)}\,\,\text{$f\in L^1(\Omega)$ if $\theta\geq 1$}\quad\text{or}\quad
\mathrm{ii)}\,\,\text{$f\in L^m(\Omega)$ for some $m > 1$ if $\theta < 1$}.$$

 \item[(H)$_g$] $g:[0,+\infty)\to [0,+\infty)$ is a  continuous function such that $g(0)=0$; moreover, in the case when $\theta < 1$, we further suppose that
\vspace{0.1cm}
\begin{equation}\tag{g1}
\exists\,\,q \geq \dfrac{1-m\theta}{m-1},\,\, \exists\,\,\nu, s_1>0: g(s) \geq \nu \, s^q, \textrm{ for every } s \geq s_1.
\end{equation}
\end{itemize}
We notice that the choices $h(s)= s^{-\gamma}$ (with $\gamma\geq 0$), $g(s) = s^q$ are both admissible, provided that $q$ satisfies the following bound
$$q\geq \frac{1-m\gamma}{m-1}$$
(as, in this case, we have $\theta = \gamma$). With this choice of 
$g$ and $h$, problem \eqref{eq:Problem} reduces to the following one,
which we will call the {\em model problem}:
\begin{equation}\label{eq:ModelProblem}
\left\{ \begin{array}{rl}
-\Delta u + (-\Delta)^s u + u^q = \dfrac{f(x)}{u^{\gamma}} & \textrm{ in } \Omega,\\
u>0 & \textrm{ in } \Omega,\\
u=0 & \textrm{ in } \mathbb{R}^{n}\setminus \Omega,
\end{array}\right.
\end{equation}
Before entering into the details of the content of the paper, let us introduce the notion of solution we will work with. 
\begin{definition}\label{def:Distributional_Solution}
Suppose that (H)$_h$,\,(H)$_f$ and (H)$_g$ hold. A function $u \in W^{1,1}(\mathbb{R}^n)$ is a \emph{distributional solution} of problem \eqref{eq:Problem} if\begin{itemize}
\item[i)] $u|_\Omega\in W_0^{1,1}(\Omega)$ and $u\equiv 0$ a.e.\,on $\mathbb{R}^n\setminus\Omega$;
\item[ii)] $u>0$ in $\Omega$;
\vspace{0.1cm}

\item[iii)] $g(u), h(u)f\in L^{1}_{\mathrm{loc}}(\Omega)$;
\vspace{0.1cm}

\item[iv)] for every $\varphi \in C^{\infty}_{0}(\Omega)$ it holds that
\begin{equation} \label{eq:defDistributionalSolByParts}
\begin{aligned}
\int_{\Omega}\nabla u \cdot \nabla \varphi \, dx &+ \iint_{\mathbb{R}^{2n}}\dfrac{(u(x)-u(y))(\varphi(x)-\varphi(y))}{|x-y|^{n+2s}}\, dx dy + \int_{\Omega}g(u)\varphi \, dx \\
&= \int_{\Omega}h(u)f \varphi\, dx.
\end{aligned}
\end{equation}
\end{itemize}
\end{definition}
While we refer to Section \ref{sec.Prel} for a detailed discussion of the well-posedness of the above definition (and for the introduction of an adequate functional setting for the operator $\LL$), here we limit ourselves to point out that the need of a \emph{globally defined} fun\-ction 
$u$ in Definition \ref{def:Distributional_Solution} is motivated by the non-local nature of $\LL$.

On the other hand, the assumed \emph{global regularity} of $u$  
(namely, $u\in W^{1,1}(\mathbb{R}^n)$) is a simple consequence of the fact that 
$$\text{($u|_\Omega\in W_0^{1,1}(\Omega)$ and $u\equiv 0$ a.e.\,in $\R^n\setminus\Omega$})\,\,\Longrightarrow\,\,u\in W^{1,1}(\mathbb{R}^n).$$

 We are now ready to state or main results: the first is an existence result, while the second concerns the uniqueness of the solutions.\begin{theorem}\label{thm:Existence}
Suppose that \emph{(H)}$_h$,\,\emph{(H)}$_f$ and \emph{(H)}$_g$ hold. Then, there exists a solu\-tion $u$ \emph{(}in the sense of Definition \ref{def:Distributional_Solution}\emph{)} of \eqref{eq:Problem} such that
\begin{equation*}
u|_\Omega \in H_0^1(\Omega) \quad \textrm{and} \quad g(u)u\in L^{1}(\Omega).
\end{equation*}
\end{theorem}

\begin{theorem}\label{thm:Uniqueness}
Suppose that \emph{(H)}$_h$,\,\emph{(H)}$_f$ and \emph{(H)}$_g$ hold. Moreover, assume that 
\begin{itemize}
\item[i)]	$h$ is non increasing;
\item[ii)]  $g$ is non decreasing.
\end{itemize}
 Then there exists at most one solution (in the sense of Definition \ref{def:Distributional_Solution}) of \eqref{eq:Problem} such that
\begin{equation*}
u|_\Omega \in H_0^1(\Omega) \quad \textrm{and} \quad g(u)\in L^{1}(\Omega).
\end{equation*}
\end{theorem}
We explicitly highlight that, by applying  Theorems \ref{thm:Existence}-\ref{thm:Uniqueness}
to the \emph{model problem} \eqref{eq:ModelProblem},
we obtain the following \emph{existence and uniqueness} result.
\begin{corollary} \label{cor:ModelProblemEU}
	Let $0\leq \gamma\leq 1$, and let $f$ satisfy
	\emph{(H)}$_f$ (with $\theta = \gamma$). According to assumption \emph{(H)}$_g$,  let $q\in\mathbb{R}$ be such that
	$$q\geq \max\{0,q_{\gamma,m}\}, \quad\text{where $q_{\gamma,m}
	= \frac{1-m\gamma}{m-1}$}$$
	(with the convention that $q_{m,\gamma} = -1$ if $\gamma = 1$).
	
	Then, there exists a unique solution $u$ of problem \eqref{eq:ModelProblem} such that
	$$u|_\Omega\in H_0^1(\Omega)\quad\text{and}\quad u\in L^{q+1}(\Omega).$$
\end{corollary}
\medskip

Let us now comment on the proof of Theorem \ref{thm:Existence}. We follow a quite classical scheme as already performed in \cite{Oliva} with a few significant issues. We first consider a family of approximating problems (see \eqref{eq:Approx_Problem}). Usually, the existence of a weak solution to such problems descend from classical Leray-Lions theory \cite{LeLi}, which cannot be used for free in our case due to the presence of the nonlocal term. This is the reason why we show that the functional analytic properties of the operator $\mathcal{L}$ are anyway enough to mimic the classical Leray-Lions existence theorem, see Theorem \ref{thm:LerayLions}. The next step in the approximation scheme consist in proving suitable apriori estimates, and this is done in Lemma \ref{lem:propun}. The final step is a limiting procedure. Most of the terms can be treated as in \cite{Oliva} except for one, see \eqref{eq:limBndelta}. In order to show the validity of \eqref{eq:limBndelta}, we test the equation with a suitable function and therefore the nonlocal part comes heavily into play. We stress here that this approach has never been used in the purely nonlocal case and indeed we reach the desired conclusion with the help of an interpolation inequality which work only because of the presence of the local term. We also stress that in the case $s \in (1/2,1)$ we also need the assumption \eqref{eq:assumptionMeasurefzero} to let the argument work.\\
\noindent It is worth to spend a further comment on the line of proof we chose to follow. In \cite{DeCOl}, the authors avoided the use of the test-function argument described before, by using a strong maximum principle for distributional solutions proved in \cite{Vazquez}. An analogous result is currently not available for local-nonlocal operators. Nevertheless, we showed in the Appendix \ref{appendix} that paying a price on the growth of the absortion term, and asking for more regularity on the datum $f$, we can deal with classical solutions and then prove a strong maximum principle. We stress that the at most critical growth asked on the absorption term is kind of natural if one looks for classical solutions, at least in the case of $-\Delta$, thanks to a famous result by Brezis and Kato \cite{BrezisKato}.

\medskip

The last result of the paper concern the model problem \eqref{eq:ModelProblem}. In this case, by adapting an idea found in \cite{BCT}, we first establish a pointwise comparison \`a la Talenti and then we use it to provide a gain of summability of the solutions which is the content of the following

\begin{theorem}\label{thm:MoreSummability}
	Let the assumptions of Corollary \ref{cor:ModelProblemEU} be satisfied, and let $u$ be the unique solution of problem \eqref{eq:ModelProblem}.
	Then, the following assertions hold.
	
	\begin{itemize}
		\item[(\emph{i})] if $1<m<\frac{n}{2}$, then we have
		\begin{equation}\label{coroll-1}
		\begin{split}
			\|u\|_{L^p(\Omega)} &\leq  \left(\frac{\gamma+1}{n^2 \omega_n^{2 / n}}\right)^{\gamma+1}\left(\frac{n(m-1)}{n-2 m}\right)^{\frac{p-1}{p(\gamma+1)}}\times \\
			&\qquad\times \left(\frac{\Gamma\left(\frac{n}{2}\right)}{\Gamma\left(\frac{n}{2 m}\right) \Gamma\left(\frac{n m-n+2 m}{2 m}\right)}\right)^{\frac{2}{n(\gamma+1)}}\|f\|_{L^m(\Omega)}^{\frac{1}{\gamma+1}},
			\end{split}
		\end{equation}
		where $p=\frac{n m(\gamma+1)}{n-2 m}$;
		\medskip
		
		\item[(\emph{ii})] if $m>\frac{n}{2}$, then we have
		\begin{equation}\label{coroll-2}
			\|u\|_{L^{\infty}(\Omega)} \leq\left(\frac{\gamma+1}{n^2 \omega_n^{2 / n}}\Big[\frac{m}{m-1}\left(\frac{n(m-1)}{2 m-n}\right)^{\frac{m-1}{m}}\Big]|\Omega|^{\frac{2 m-n}{n m}}\right)^{\frac{1}{\gamma+1}}\|f\|_{L^m(\Omega)}^{\frac{1}{\gamma+1}}\,;
		\end{equation}
		\item[(\emph{iii})] if $m=\frac{n}{2}$, then $u$ belongs to the Orlicz space $L_{\exp t^{\frac{n(\gamma+1)}{n-2}}}(\Omega)$ and
		\begin{equation}\label{coroll-3}
			\sup _{s \in(0,|\Omega|)} \frac{u^\ast(s)}{\left(\log \frac{|\Omega|}{s}\right)^{\frac{n-2}{n(\gamma+1)}}} \leq\left(\frac{\gamma+1}{n^2 \omega_n^{2 / n}}\right)^{\frac{1}{\gamma+1}}\|f\|_{L^{n / 2}(\Omega)}^{\frac{1}{\gamma+1}}.
		\end{equation}
	\end{itemize}
\end{theorem}
\begin{remark} \label{rem:StimaSignificativa}
Let the assumptions of Corollary \ref{cor:ModelProblemEU} be satisfied \emph{with $q = q_{\gamma,m}$},
and let $u$ be the unique solution of problem \eqref{eq:ModelProblem}, further satisfying $$\mathrm{a)}\,\,u|_\Omega\in H_0^1(\Omega)\quad\text{and}\quad 
\mathrm{b)}\,\,u\in L^{q+1}(\Omega).$$
On account of b), estimate \eqref{coroll-1} (which holds true in the case $1<m<n/2$) is a true
\emph{gain of integrability}	 for $u$ if and only if
$$q_{\gamma,m}+1< \frac{nm(\gamma+1)}{n-2m}\,\,\Longleftrightarrow\,\,
m >m_\gamma =\frac{2n}{n(\gamma+1)-2(1-\gamma)} $$
If, instead, $m\leq m_\gamma$, the absorption term $g(s) = s^q$ (where we choose $q = q_{\gamma,m}$, the smallest admissible value of $q$ ensuring the existence of a solution) produces a 
\emph{stronger regularizing effect} then that of Theorem \ref{thm:MoreSummability}.
\end{remark}

We also stress the technique adopted in \cite{BO} to prove the same gain of summability without the absorption term consists in a test-function argument at the level of the approximating problems that is preserved in the limit, so getting it for the solution as well. We decided to follow a rather indirect path by proving first Theorem \ref{thm:pointwise_talenti} because this result (which can be proved to hold also in the nonsingular case $\gamma=0$) provides a possible improvement w.r.t.\,\cite{bahrouni}, where the author obtained a mass comparison principle (akin to the purely nonlocal one obtained in \cite{FeVo}). It has to be mentioned that the comparison problem considered in \cite{bahrouni} is actually purely nonlocal, differently from our choice here.

\medskip

We close the Introduction recalling that problem \eqref{eq:Problem} naturally falls in the framework of the so called {\em singular} problems. We list a few seminal papers which gave rise to the study of such PDEs (with leading term $-\Delta$), see e.g. \cite{Fulks, Stuart, CRT, LazMck, LairShaker} and it is still an active topic of research, see e.g. \cite{BO, CaDe, CGS, CMSS, CMST, OP, OP2, Oliva}. We refer to the recent survey \cite{OP3} for a very nice and detailed introduction to the topic and to the physical motivations that led to the study of this kind of PDEs. The literature concerning singular problems driven by $(-\Delta)^s$ is also pretty vast: we refer e.g. to \cite{BdBMP, ArNgRa, HuNg}. Finally, we refer to \cite{AnGa, Garain, GarainUkhlov, BiagiVecchi, ArRa, BhGh, BiGa, GKK, Biroud} for several results on singular problems driven by the local-nonlocal operator $\mathcal{L}$. 
\medskip

\noindent -\,\,\emph{Plan of the paper}. The paper is organized as follows: in Section \ref{sec.Prel} we fix the notations used in the paper and we recall all the relevant results needed in what follows. We also prove Theorem \ref{thm:LerayLions} whose importance in our argument has been described before. In Section \ref{sec:Existence_and_uniqueness} we prove Theorem \ref{thm:Existence} and Theorem \ref{thm:Uniqueness}, while the proof of Theorem \ref{thm:MoreSummability} is the content of Section \ref{sec:talenti}, along with the pointwise comparison briefly mentioned before.
\bigskip

\noindent-\,\,\textbf{Conflict of interest.} The author states no conflict of interest.

\noindent-\,\,\textbf{Data availability statement.} There are no data associated with this research.

\section{Preliminaries}\label{sec.Prel}
\noindent {\bf Notations.} Throughout the paper, we tacitly
exploit all the notation listed below; we thus refer the Reader to this list
for any non-standard notation encountered.
\begin{itemize}
\item $\Omega \subset \mathbb{R}^{n}$ is an open and bounded set with smooth enough boundary $\partial \Omega$.
\item Given any set $E$, $|E|$ denotes its the $n$-dimensional Lebesgue measure.
\item Given any two vectors $v,w \in \mathbb{R}^{n}$, we denote by
$v\cdot w$ the usual scalar product.
\item Given any $p \in [1,n)$ we denote by $p^{*}$ the associated Sobolev exponent (with respect to Euclidean space $\mathbb{R}^n$), that is
$$p^{*}  := \frac{np}{n-p}.$$
\item Given any $r \in (1,+\infty)$, we denote by $r'$ the conjugate exponent of $r$ in the usual H\"older inequality, that is, 
$$r'  := \frac{r}{r-1}.$$
\item Given a Banach reflexive space $V$ with dual space $V'$, we denote by 
$\langle \cdot,\cdot\rangle$
the duality product. 
\item Given a measurable set $E$, we denote by $\chi_{E}$ its characteristic function.
\item Given any function $v$, we denote the positive and negative part of $v$ by 
\begin{equation}\label{eq:def_positive_negative_part}
v^{+}:= \max\{v,0\} \quad \textrm{ and } \quad v^{-}:= -\min\{v,0\}.
\end{equation}
\item For every fixed $k>0$, we define the truncation functions as follows:
\begin{align}
& T_{k}(s):= \max \{ -k, \min\{s,k\}\} \quad (s \in \mathbb{R}) \label{eq:def_Tk} \\
& G_{k}(s):= (|s|-k)^{+} \, \mathrm{sign}(s) \quad (s\in \mathbb{R}) \label{eq:def_Gk}
\end{align}
\noindent where $\mathrm{sign}(\cdot)$ denotes the usual {\it sign function}.
\item  For every fixed $\delta,k>0$, we also define the following functions
\begin{equation}\label{eq:def_V_delta_k}
V_{\delta,k}(s):= \left\{\begin{array}{lr}
1 & s \leq k,\\
\tfrac{k+\delta-s}{\delta} & k<s<k+\delta,\\
0 & s \geq k+\delta,
\end{array}\right.
\end{equation}
\noindent and 
\begin{equation}\label{eq:def_S_delta_k}
S_{\delta,k}(s):= 1 - V_{\delta,k}(s).
\end{equation}
\end{itemize}
\medskip

 \noindent Taking for granted all the notation listed above, we now turn to spend a few words on the \emph{natural functional setting} associated with the mixed operator
 $$\mathcal{L} = -\Delta+(-\Delta)^s.$$
 Actually, since a key ingredient in our approach will be a \emph{Leray-Lions-type} result for $\mathcal{L}$, and since this result naturally applies
 to \emph{quasilinear operators}, we consider the quasilinear
 generalization of $\mathcal{L}$, that is,
 $$\mathcal{L}_{p} = -\Delta_p+(-\Delta)_p^s,$$
 where $-\Delta_p  u = \mathrm{div}(|\nabla u|^{p-2}\nabla u)$ is the usual $p$-Laplace operator.
 \bigskip
 
 \noindent\textbf{1) The fractional $p$-Laplacian.} Let $p\in(1,\infty)$ and $s\in (0,1)$ be fixed
 \emph{once and for all}. If $u:\R^n\to\R$ is a measurable function, the \emph{fractional $p$-La\-pla\-cian} (of order $s$) of $u$
at a point $x\in\R^n$ is defined (up to a
\emph{nor\-ma\-lization constant} $C_{n,s,p} > 0$ that we neglect, and setting $B_p(a) = |a|^{p-2}a$) as follows
\begin{align*}
 (-\Delta)_p^s u(x) & = 2\,\mathrm{P.V.}\int_{\R^n}\frac{B_p(u(x)-u(y))}{|x-y|^{n+ps}}\,dy\\
 & = 2 \lim_{\varepsilon\to 0^+}\int_{\{|x-y|\geq\varepsilon\}}
 \frac{B_p(u(x)-u(y))}{|x-y|^{n+ps}}\,dy,
\end{align*}
provided that the above limit \emph{exists and is finite}.
As it is reasonable to expect, for $(-\Delta)_p^s u(x)$ to be well-defined one needs
to impose suitable \emph{growth conditions} on the functions $u$, both when $y\to\infty$ and
when $y\to x$. 
In this perspective we have the following
 proposition, where we employ the notation
\begin{equation}\label{eq:spaceL1s}
 {L}^{s,p}(\R^n) := 
 \Big\{f\in L^1_{\mathrm{loc}}(\R^n):\,\|f\|_{p,s} := \int_{\R^n}\frac{|f(x)|^{p-1}}
  {1+|x|^{n+ps}}\,dx<\infty\Big\}.
\end{equation}
\begin{proposition}[{See \cite[Lemmas 3.6 and 3.8]{KKL}}] \label{prop:welldefDeltas}
Let $\Oo\subseteq\R^n$ be an open set, and
 let $u\in\LL^{s,p}(\R^n)\cap C^{2}(\Oo)$.
 Then, if we assume that $sp<p-1$, we have
 $$\exists\,\,(-\Delta)_p^s u(x) = 2\,\int_{\R^n}\frac{B_p(u(x)-u(y))}{|x-y|^{n+ps}}\,dy\quad
 \text{for all $x\in\Oo$}. $$
Moreover, $(-\Delta)_p^su\in C(\Oo)$.
\end{proposition}
Even if Proposition \ref{prop:welldefDeltas} provides an easy
sufficient condition for the well-posed\-ness of the \emph{pointwise definition}
of $(-\Delta)^s_p u(x)$, it is well-known
that the develop of a \emph{pointwise theory} for (local or nonlocal)
quasilinear operators presents intrinsic difficulties
(in particular when $1<p<2$); for this reason,
we now review some basic
facts concerning the \emph{Weak Theory} of $(-\Delta)_p^s$.
\vspace*{0.1cm}

To begin with we recall that, if $\mathcal{O}\subseteq\R^n$ is an arbitrary open set,
the {fractional $p$-Laplacian} $(-\Delta)_p^s$ is related
(essentially via the Euler-Lagrange equation)
to the \emph{fractional Sobolev space} $W^{s,p}(\Oo)$, which is defined as follows:
$$W^{s,p}(\Oo) := \Big\{u\in L^p(\Oo):\,[u]^p_{p,s,\Oo} = \iint_{\Oo\times\Oo}
\frac{|u(x)-u(y)|^p}{|x-y|^{n+ps}}\,dx\,dy < \infty\Big\}.$$
While we refer to \cite{LeoniFract} for a detailed introduction
on fractional Sobolev spaces, here we list the few basic properties of $W^{s,p}(\Oo)$
we will exploit in this paper.
\begin{itemize}
 \item[a)] $W^{s,p}(\Oo)$ is a real Banach space, with the norm
 $$\|u\|_{s,p,\Oo}
 := \Big(\|u\|^p_{L^p(\Oo)}+\iint_{\Oo\times\Oo}\frac{|u(x)-u(y)|^p}{|x-y|^{n+ps}}\,dx\,dy\Big)^{1/p}\qquad
 (u\in W^{s,p}(\Oo)).$$ 
 
 \item[b)] $C_0^\infty(\Oo)$ is a \emph{linear subspace of $W^{s,p}(\Oo)$}; in addition,
 in the particular case when $\Oo = \R^n$, we have that
  $C_0^\infty(\R^n)$ is \emph{dense} in $W^{s,p}(\R^n)$.
 \vspace*{0.1cm}
 
 \item[c)] If $\Oo = \R^n$ or if $\Oo$ has \emph{bounded boundary $\partial\Oo\in C^{0,1}$},
 we have the \emph{continuous embedding} $W^{1,p}(\Oo) \hookrightarrow W^{s,p}(\Oo)$,
 that is, 
 \begin{equation} \label{eq:H1embeddingHs}
  \iint_{\Oo\times\Oo}\frac{|u(x)-u(y)|^p}{|x-y|^{n+ps}}\,dx\,dy \leq 
  \mathbf{c}\,\|u\|^p_{W^{1,p}(\Oo)}\quad\text{$\forall\,\,u\in W^{1,p}(\Oo)$},
 \end{equation}
 for some constant $\mathbf{c} = \mathbf{c}(n,p,s) > 0$.
  In particular, if $\Oo\subseteq\R^n$ is a
  \emph{bounded open set} (with no regularity as\-sump\-tions on $\partial\Oo$) and
  if $u\in W_0^{1,p}(\Oo)$, set\-ting $\hat{u} = u\cdot\mathbf{1}_\Oo\in W^{1,p}(\R^n)$ we have
  \begin{equation} \label{eq:H01embeddingHs}
  \iint_{\R^{2n}}\frac{|\hat{u}(x)-\hat{u}(y)|^p}{|x-y|^{n+ps}}\,dx\,dy \leq 
  \beta\,\int_\Oo|\nabla u|^p\,dx,
 \end{equation}
 where $\beta > 0$ is a suitable constant depending on $n,s$ and on $|\Omega|$. Here
 and throughout,  $|\cdot|$ denotes the $n$-dimensional Lebesgue measure.
\end{itemize}
As usual, we also say that $u\in W^{s,p}_{\mathrm{loc}}(\Oo)$ if
$$\text{$u\in W^{s,p}(\mathcal{V})$ for every open set $\mathcal{V}\Subset\Oo$}.$$

The relation between the space $W^{s,p}(\Oo)$ and the fractional $p$-Laplacian $(-\Delta)_p^s$
is rooted in the following \emph{fractional integration-by-parts formula}:
assuming for a moment that $sp < p-1$,
let 
$\Oo\subseteq\R^n$ be an open set (bounded or not), and let
$$u\in {L}^{s,p}(\R^n)\cap C^2(\mathcal{O});$$
given any test function $\varphi\in C_0^\infty(\Oo)$, 
we then have
\begin{equation} \label{eq:intbypartsregular}
\begin{split}
 \int_{\Oo}(-\Delta)_p^su\,\varphi\,dx & = \iint_{\R^{2n}}\frac{B_p(u(x)-u(y))
  (\varphi(x)-\varphi(y))}{|x-y|^{n+2s}}\,dx\,dy,
 \end{split}
\end{equation}
and the double integral in the right-hand side of the above
\eqref{eq:intbypartsregular} \emph{is finite}
(notice that, since $u\in {L}^{s,p}(\R^n)\cap C^2(\mathcal{O})$
and since $sp < p-1$, from
Proposition \ref{prop:welldefDeltas} we 
know that $(-\Delta)_p^su$ is well-defined pointwise in $\mathcal{O}$,
and $(-\Delta)_p^su\in C(\mathcal{O})$).
\vspace*{0.1cm}

On the other hand, it is not difficult to recognize that
the finitness of such a double integral
is guaranteed \emph{even if $sp\geq p-1$, and if we only assume that}
$$u\in {L}^{s,p}(\R^n)\cap W_{\mathrm{loc}}^{s,p}(\Oo)$$
(notice that this clearly holds if $u\in {L}^{s,p}(\R^n)\cap C^2(\mathcal{O})$,
thanks to identity \eqref{eq:H1embeddingHs}):
indeed, under this more general assumption,
if $\varphi\in C_0^\infty(\Oo)$ and if $\mathcal{V}\Subset \Oo$ is an open
set such that $K = \mathrm{supp}(\varphi)\subseteq \mathcal{V}$, we have
\begin{equation} \label{eq:contDeltasu}
  \begin{split}
   \iint_{\R^{2n}}&\frac{|u(x)-u(y)|^{p-1}
 |\varphi(x)-\varphi(y)|}{|x-y|^{n+ps}}\,dx\,dy 
 \\
 &\qquad = 
 \iint_{\mathcal{V}\times\mathcal{V}}\frac{|u(x)-u(y)|^{p-1}|\varphi(x)-\varphi(y)|}{|x-y|^{n+2s}}\,dx\,dy
 \\
 &\qquad\quad-2\,\iint_{(\R^n\setminus\mathcal{V})\times K}
 \frac{|u(x)-u(y)|^{p-1}|\varphi(y)|}{|x-y|^{n+2s}}\,dx\,dy \\
 & \qquad \leq C\big(1+\mathrm{dist}(K,
 \R^n\setminus\mathcal{V})^{-ps}\big)\|\varphi\|_{s,p,\mathcal{V}},
  \end{split} 
 \end{equation}
 (here, the constant $C>0$ depends on $n,s,p$ and on $u$).
\vspace*{0.05cm}

 As a consequence of this fact, it is thus natural to define the 
\emph{fractional $p$-Lapla\-cian} $(-\Delta)_p^s u$
of a function $u\in {L}^{s,p}(\R^n)\cap W_{\mathrm{loc}}^{s,p}(\Oo)$ as the
\emph{linear functional} acting on the space $C_0^\infty(\Oo)$ as follows
\begin{equation} \label{eq:fractionalLapWeak}
\begin{split}
  \langle (-\Delta)^s_p u,\varphi\rangle & = 
 \iint_{\R^{2n}}\frac{B_p(u(x)-u(y))(\varphi(x)-\varphi(y))}{|x-y|^{n+ps}}\,dx\,dy.
 \end{split}
 \end{equation}
 On account of \eqref{eq:contDeltasu} (and observing that 
 $\|\varphi\|_{s,p,\mathcal{V}}
 \leq c\,\|\varphi\|_{C^1(\overline{\mathcal{V}})}$ for some absolute constant $c > 0$,
 see \eqref{eq:H01embeddingHs}),
 we easily see that 
$$(-\Delta)^s_p u\in\mathcal{D}'(\Oo).$$
 In particular, if $u\in W^{s,p}(\R^n)\subseteq L^{s,p}(\R^n)$,
 by the density of $C_0^\infty(\R^n)$
   in $W^{s,p}(\R^n)$ we derive that 
   $W_0^{s,p}(\R^n) = W^{s,p}(\R^n),$
   and thus
  $$(-\Delta)_p^s u\in (W^{s,p}(\R^n))'.$$ 
  In this case, we also have
  $$\frac{\mathrm{d}}{\mathrm{d}t}\Big|_{t = 0}
  \Big(\frac{1}{p}\,[u+tv]^p_{s,p,\R^n}\Big) = (-\Delta)_p^s u(v)\quad
  \text{for all $v\in W^{s,p}(\R^n)$}.$$
  \begin{remark}\label{rem:HigherRegDeltaps}
   Let $\Oo\subseteq\R^n$ be an open set (bounded or not), and let 
   $$u\in {L}^{s,p}(\R^n)\cap W_{\mathrm{loc}}^{s,p}(\Oo);$$ 
   moreover, let $(-\Delta)^s_p u
   \in\mathcal{D}'(\Oo)$
   be the distribution on $\Oo$ defined by formula
   \eqref{eq:fractionalLapWeak}. 
   
   Given \emph{any bounded open set $\Omega$ such that $\Omega\Subset\Oo$},
   by combining 
   \eqref{eq:contDeltasu} (where we choose the open set $\mathcal{V}$
   in such a way that $\Omega\Subset\mathcal{V}\Subset\Oo$)
   with the con\-ti\-nuous embedding in \eqref{eq:H01embeddingHs}, we derive that
   \begin{equation*}
  \begin{split}
   \iint_{\R^{2n}}&\frac{|u(x)-u(y)|^{p-1}
 |\varphi(x)-\varphi(y)|}{|x-y|^{n+ps}}\,dx\,dy 
 \\
 & \qquad \leq C\big(1+\mathrm{dist}({\Omega},\R^n\setminus\mathcal{V})^{-ps}\big)
 \Big(\int_\Omega|\nabla \varphi|^p\,dx\Big)^{1/p},
  \end{split} 
 \end{equation*}
 \emph{for every function $\varphi\in C_0^\infty(\Omega)$} (and for some constant $C > 0$
 depending on $n,s,p$ and on $u$). As a consequence, we deduce that $(-\Delta)^s_p u$
 can be extended to a \emph{linear and continuous operator on $W_0^{1,p}(\Omega)$}, that is,
 $$(-\Delta)^s_p u\in W^{-1,p'}(\Omega) = \big(W_0^{1,p}(\Omega)\big)'.$$
 More precisely, we have
 \begin{align*}
  \langle (-\Delta)^s_p u,\varphi\rangle & = \iint_{\R^{2n}}\frac{B_p(u(x)-u(y))
 (\hat\varphi(x)-\hat\varphi(y))}{|x-y|^{n+ps}}\,dx\,dy
 \\
 & = \iint_{\mathcal{V}\times\mathcal{V}}\frac{B_p(u(x)-u(y))
 (\hat\varphi(x)-\hat\varphi(y))}{|x-y|^{n+2s}}\,dx\,dy
 \\
 &\qquad\quad-2\,\iint_{(\R^n\setminus\mathcal{V})\times\Omega}
 \frac{B_p(u(x)-u(y))\hat\varphi(y)}{|x-y|^{n+2s}}\,dx\,dy,
 \end{align*}
 for every $\varphi\in W_0^{1,p}(\Omega)$ (here we have set, as above,
 $\hat\varphi = \varphi\cdot\mathbf{1}_\Omega\in W^{1,p}(\R^n)$).
  \end{remark}
\medskip

\noindent\textbf{2) The space $\mathcal{X}^{1,p}(\Omega)$}.
Now we have pro\-per\-ly defined the fractional $p$\--La\-pla\-cian
 of a \emph{non-regular function} (as a distribu\-tion),
and taking into account Remark \ref{rem:HigherRegDeltaps}, we
introduce a suitable \emph{function space} which will allow us to define an
adequate \emph{Weak Theory} for $\mathcal{L}_{p,s}$ 
(and for the study of problem \eqref{eq:Problem}).
\vspace{0.1cm}

Let then $\Omega\subseteq\R^n$ be a bounded open set.  In view of the usual definition of the (local) Sobolev space $H_0^1(\Omega)$,
  and taking into account \eqref{eq:fractionalLapWeak}, we 
  define 
  $\mathcal{X}^{1,p}(\Omega)$ as the com\-ple\-tion
  of $C_0^\infty(\Omega)$ w.r.t.\,the \emph{global norm} 
  $$\rho_p(u) := \left(\||\nabla u|\|^p_{L^p(\R^n)}+[u]^p_{s,p,\R^n}\right)^{1/p},
  \qquad u\in C_0^\infty(\Omega).$$
  Concerning this space $\mathcal{X}^{1,p}(\Omega)$ we first observe that,
  since we assuming $1<p<\infty$, we clearly have that it  is a \emph{Banach space}; most importantly, since $\Omega$ is bounded, by combining 
    \eqref{eq:H1embeddingHs} with the Poincar\'e
    inequality we infer that
    \begin{equation*}
     \vartheta^{-1}\|u\|_{W^{1,p}(\R^n)}\leq \rho_p(u)\leq \vartheta\|u\|_{W^{1,p}(\R^n)}\qquad
    \text{for every $u\in C_0^\infty(\Omega)$},
    \end{equation*}
    where $\vartheta > 1$ is a suitable constant depending on $n,s,p$ and on $|\Omega|$.
    Thus, $\rho_p(\cdot)$ and the full $W^{1,p}$-norm in $\R^n$
   are \emph{actually equivalent} on the space $C^\infty_0(\Omega)$, so that
   \begin{equation} \label{eq:defX12explicit}
   \begin{split}
    \mathcal{X}^{1,p}(\Omega) & = \overline{C_0^\infty(\Omega)}^{\,\,\|\cdot\|_{W^{1,p}\R^n)}} \\
    & = \{u\in W^{1,p}(\R^n):\,\text{$u|_\Omega\in W_0^{1,p}(\Omega)$ and 
    $u\equiv 0$ a.e.\,in $\R^n\setminus\Omega$}\}.
    \end{split}
   \end{equation}
   As a consequence, the following properties hold:
   \begin{align}
   	(1)\,\,&\text{If $u\in W_0^{1,p}(\Omega)$, then $u\cdot\mathbf{1}_\Omega\in \mathcal{X}^{1,p}(\Omega)$}; \label{eq:ExtensionW0X1p} \\
   	(2)\,\,&\mathcal{X}^{1,p}(\Omega)\subseteq W^{1,p}(\mathbb{R}^n)\subseteq
   	L^{s,p}(\mathbb{R}^n)\cap W^{s,p}(\mathbb{R}^n). \label{eq:X1pWsLps}
   \end{align}
   Since throughout this paper we are mainly interested in the \emph{case $p = 2$}, we collect in the next remark some properties of the space $\mathcal{X}^{1,2}(\Omega)$ which will be repeatedly used in the sequel.
 To avoid cumbersome notation, we simply set
 $$\rho(u) = \rho_2(u)\quad\text{and}\quad [u]_{s,\mathbb{R}^n} = [u]_{s,2,\mathbb{R}^n}.$$
   \begin{remark}[Properties of the space $\mathcal{X}^{1,2}(\Omega)$] \label{rem:spaceX12}
    The following assertions hold.
        \begin{enumerate}
     \item[1)] Since both $H^1(\R^n)$ and $H_0^1(\Omega)$ are \emph{closed} under the
     maximum/minimum o\-pe\-ra\-tion, it is readily seen that
     $$u^{\pm}\in \mathcal{X}^{1,2}(\Omega)\quad\text{for every $u\in\mathcal{X}^{1,2}(\Omega)$},$$
     \noindent where $u_{+}$ and $u_{-}$ have been defined in \eqref{eq:def_positive_negative_part}.
     \item[2)] On account of \eqref{eq:H01embeddingHs}, for every
     $u\in\mathcal{X}^{1,2}(\Omega)$ we have
     \begin{equation} \label{eq:X12HsRn}
    [u]_{s,\R^n}^2 = 
     \iint_{\R^{2n}}\frac{|u(x)-u(y)|^2}{|x-y|^{n+2s}}\,dx\,dy \leq\beta\int_\Omega|\nabla u|^2\,dx.
   \end{equation}
   As a consequence,
   the norm $\rho$ is \emph{glo\-bal\-ly equivalent} on $\mathcal{X}^{1,2}(\Omega)$ to the 
   $H_0^1$-no\-rm: in fact, by \eqref{eq:X12HsRn} there exists a constant $\Theta
   = \Theta_{n,s} > 0$ such that
   \begin{equation} \label{eq:equivalencerhoH01}
    \||\nabla u|\|_{L^2(\Omega)}\leq \rho(u)\leq \Theta\||\nabla u|\|_{L^2(\Omega)}
    \quad\text{for every $u\in\mathcal{X}^{1,2}(\Omega)$}.
   \end{equation}
   
   \item[3)] By the (local) Sobolev inequality, for every $u\in\mathcal{X}^{1,2}(\Omega)$ 
   we have
   \begin{align*}
    S_n\|u\|_{L^{2^\ast}(\Omega)}^2 & = S_n\|u\|_{L^{2^\ast}(\R^n)}^2
    \leq \int_{\R^n}|\nabla u|^2\,dx\leq \rho(u)^2.
   \end{align*}
   This, together with H\"older's inequality
   (recall that $\Omega$ is \emph{bounded}), proves the \emph{continuous embedding}
   $\text{$\mathcal{X}^{1,2}(\Omega)\hookrightarrow L^{m}(\Omega)$ for every $1\leq m\leq 2^\ast$}.$
   \vspace*{0.1cm}
   
   \item[4)] By combining \eqref{eq:equivalencerhoH01} with the \emph{compact embedding} of 
   $H_0^1(\Omega)\hookrightarrow L^m(\Omega)$ (holding true for every $1\leq m < 2^\ast$),
   we derive that also the embedding
   $$\text{$\mathcal{X}^{1,2}(\Omega)\hookrightarrow L^{m}(\Omega)$ is compact
   for every $1\leq m< 2^\ast$}.$$
   As a consequence, if $\{u_k\}_k$ is a bounded sequence in $\mathcal{X}^{1,2}(\Omega)$, it is possible
   to find a (unique) function $u\in\mathcal{X}^{1,2}(\Omega)$ such that (up to a sub-sequence)
   \begin{itemize}
    \item[a)] $u_n\to u$ weakly in $\mathcal{X}^{1,2}(\Omega)$;
    \item[b)] $u_n\to u$ \emph{strongly} in $L^m(\Omega)$ for every $1\leq m < 2^\ast$;
    \item[c)] $u_n\to u$ pointwise a.e.\,in $\Omega$.
   \end{itemize}
   Clearly, since both $u_n$ (for all $n\geq 1$) and $u$ \emph{identically vanish}
    out of $\Omega$, we can replace
   $\Omega$ with $\R^n$ in the above assertions b)-c).
    \end{enumerate}
       \end{remark}
As already anticipated, the space $\mathcal{X}^{1,p}(\Omega)$ just introduced
is naturally associated with the \emph{weak realization} of $(-\Delta)^s_p$
via Remark \ref{rem:HigherRegDeltaps}. Indeed, let $\mathcal{O}\subseteq\mathbb{R}^n$ be an arbi\-trary open set, and let
$u\in L^{s,p}(\mathbb{R}^n)\cap W^{s,p}_{\mathrm{loc}}(\mathcal{O})$.
If $\Omega\subseteq\mathbb{R}^n$ is a \emph{bounded open set} such that
$\Omega\Subset\mathcal{O}$,  from Remark \ref{rem:HigherRegDeltaps}
we know that
the distribution $$(-\Delta)^s_p u\in \mathcal{D}'(\mathcal{O})$$ can be
extended to a linear and continuous operator on $W^{1,p}_0(\Omega)$ by setting
\begin{align*}
  \langle (-\Delta)^s_p u,\varphi\rangle& = \iint_{\R^{2n}}\frac{B_p(u(x)-u(y))
 (\hat\varphi(x)-\hat\varphi(y))}{|x-y|^{n+ps}}\,dx\,dy
 \\
 & = \iint_{\mathcal{V}\times\mathcal{V}}\frac{B_p(u(x)-u(y))
 (\hat\varphi(x)-\hat\varphi(y))}{|x-y|^{n+2s}}\,dx\,dy
 \\
 &\qquad\quad-2\,\iint_{(\R^n\setminus\mathcal{V})\times\Omega}
 \frac{B_p(u(x)-u(y))\hat\varphi(y)}{|x-y|^{n+2s}}\,dx\,dy,
 \end{align*}
where
 $\hat\varphi = \varphi\cdot\mathbf{1}_\Omega\in W^{1,p}(\R^n)$. Hence, since  $\mathcal{K} = \big\{\hat{\varphi}:\,\varphi\in W_0^{1,p}(\Omega)\big\}$
 \emph{coincides with our space $\mathcal{X}^{1,p}(\Omega)$},
 and since $\rho_p(\cdot)\geq \|\cdot\|_{W_0^{1,p}(\Omega)}$, we get
 $$(-\Delta)^s_pu\in (\mathcal{X}^{1,p}(\Omega))';$$ 
 moreover, for every function $\varphi\in \mathcal{X}^{1,p}(\Omega)$ we have
\begin{equation} \label{eq:explicitDeltapsX1p}
\begin{split}
	 \langle (-\Delta)^s_p u,\varphi\rangle & = \iint_{\R^{2n}}\frac{B_p(u(x)-u(y))
 (\varphi(x)-\varphi(y))}{|x-y|^{n+ps}}\,dx\,dy
 \\
 & = \iint_{\mathcal{V}\times\mathcal{V}}\frac{B_p(u(x)-u(y))
 (\varphi(x)-\varphi(y))}{|x-y|^{n+2s}}\,dx\,dy
 \\
 &\qquad\quad-2\,\iint_{(\R^n\setminus\mathcal{V})\times\Omega}
 \frac{B_p(u(x)-u(y))\varphi(y)}{|x-y|^{n+2s}}\,dx\,dy.
\end{split}
\end{equation}
In particular, if we also have that $u\in\mathcal{X}^{1,p}(\Omega)$ (whence,
$u\in L^{s,p}(\mathbb{R}^n)\cap W^{s,p}(\mathbb{R}^n)$, see \eqref{eq:X1pWsLps}), for every $\varphi\in \mathcal{X}^{1,p}(\Omega)$
we have the estimate
\begin{equation} \label{eq:DeltapsonX1pbuono}
\begin{split}
|\langle (-\Delta)^s_p u,\varphi\rangle| & \leq
\iint_{\R^{2n}}\frac{|u(x)-u(y)|^{p-1}|\varphi(x)-\varphi(y)|}{|x-y|^{n+2s}}\,dx\,dy \\
& \leq  [u]_{s,p,\mathbb{R}^n}
[\varphi]_{s,p,\mathbb{R}^n}
\leq \rho_p(u)\,\rho_p(\varphi),
\end{split}	
\end{equation}
and thus $u\mapsto (-\Delta)^s_pu$ is continuous from $\mathcal{X}^{1,p}(\Omega)$
into its dual.

\medskip

\noindent\textbf{3) The mixed operator $\LL_{p}$ and problem \eqref{eq:Problem}.}
Taking everything said so far into account, 
we can introduce the \emph{weak realization} of $\mathcal{L}_{p}$ 
we are going to use in this paper (mainly in the case $p = 2$), and
give some comments on Definition \ref{def:Distributional_Solution}.
\vspace{0.1cm}

To begin with we recall that, if 
 $\Oo\subseteq\R^n$ is an arbitrary open set (bounded or not)
 and if $u\in W^{1,p}_{\mathrm{loc}}(\Oo)$, the \emph{$p$-Laplacian} $-\Delta_p u$
 of $u$
 is defined as the {linear operator}
 $$-\Delta_pu:C_0^\infty(\Oo)\to\R,\qquad \langle -\Delta_p u,\varphi\rangle = \int_{\Oo}|\nabla u|^{p-2}\langle
 \nabla u,\nabla \varphi\rangle\,dx.$$
Clearly, such an operator is in fact a distribution on $\mathcal{O}$. Moreover, for any bounded open set $\Omega \Subset \mathcal{O}$, the purely local nature of $-\Delta_p$ ensures that it can be extended (with the same expression) \emph{as a linear and continuous operator on $\mathcal{X}^{1,p}(\Omega)$}.
 \vspace{0.1cm}
 
 In view of these facts, and since $W^{1,p}_{\mathrm{loc}}(\Oo)\subseteq W^{s,p}_{\mathrm{loc}}(\Oo)$
 (see \eqref{eq:H1embeddingHs}), 
 for any gi\-ven function $u\in L^{s,p}(\R^n)\cap W^{1,p}_{\mathrm{loc}}(\Oo)$,
 {we can thus \emph{define $\LL_{p} u$
 as the linear o\-pe\-rator
 on the space $C_0^\infty(\Oo)$ given by the sum between $-\Delta_p u$ and $(-\Delta)^s_p u$}, that is,
 \begin{equation} \label{eq:defLpsdistrib}
 \begin{split}
  \langle \LL_{p} u,\varphi\rangle
 & = \langle -\Delta_p u,\varphi\rangle+\langle (-\Delta)^s_p u,\varphi\rangle \\
 & = \int_{\Oo}|\nabla u|^{p-2}\langle
 \nabla u,\nabla \varphi\rangle\,dx
 \\
 & \qquad+ \iint_{\R^{2n}}\frac{B_p(u(x)-u(y))(\varphi(x)-\varphi(y))}{|x-y|^{n+ps}}\,dx\,dy.
 \end{split}
 \end{equation}
 In particular, we see that
 \begin{itemize}
  \item[1)] \emph{$\LL_{p} u$ is a distribution on $\Oo$}, that is, $\LL_{p}\in\mathcal{D}'(\Oo)$;
  \vspace{0.05cm}
  
  \item[2)] given any \emph{bounded open set} $\Omega\Subset \Oo$, the distribution
  $\LL_{p} u$ can be extended 
  \emph{with the same expression} as a linear and continuous operator on $\mathcal{X}^{1,p}(\Omega)$.
 \end{itemize}
 \begin{remark} \label{rem:LuwhenuX12}
 	Let $\Omega\subseteq\mathbb{R}^n$ be a bounded open set, and let
 	 $u\in \mathcal{X}^{1,p}(\Omega)$. On account of \eqref{eq:X1pWsLps},
 	 we know that $u\in W^{1,p}(\mathbb{R}^n)\subseteq
 	 L^{s,p}(\mathbb{R}^n)$;
 	as a consequence,
 	$$\LL_{p}u\in (\mathcal{X}^{1,p}(\Omega))'$$
 	Moreover, owing to \eqref{eq:defLpsdistrib}
 	(see also \eqref{eq:DeltapsonX1pbuono}), we have
 	\begin{align*}
 	 |\langle \LL_{p} u,\varphi\rangle|
 	 & \leq \||\nabla u|\|_{L^p(\Omega)}\||\nabla \varphi|\|_{L^p(\Omega)}
 	 + [u]_{p,s,\mathbb{R}^n}\,[\varphi]_{p,s,\mathbb{R}^n} \\
 	 & \leq 2\rho_p(u)\,\rho_p(\varphi)\quad\text{for every
 	 $\varphi\in \mathcal{X}^{1,p}(\Omega)$}. 	
 	\end{align*}
 	and thus we can think of $\mathcal{L}_p$ as an operator
 	from $\mathcal{X}^{1,p}(\Omega)$ into $(\mathcal{X}^{1,p}(\Omega))'$.
 	\end{remark}
 
 Thanks to the preceding discussion (which provides a precise definition of 
  $\LL_{p}u$
 under mild assumptions on $u$),
 we spend a few words on Definition \ref{def:Distributional_Solution}.
	\medskip
	
	1)\,\,First of all we observe that, if $u:\mathbb{R}^n\to\mathbb{R}$ is a  solution of problem \eqref{eq:Problem} (in the sense of Definition
	 \ref{def:Distributional_Solution}), we \emph{cannot guarantee}
	 that $u\in H^1_{\mathrm{loc}}(\Omega)$, and thus $\LL u$
	 \emph{is not well-defined} (as a distribution on $\Omega$);
	 in other words, we \emph{cannot say} that the first two terms in the left-hand side of \eqref{eq:defDistributionalSolByParts} do coincide with 
	 $\langle \LL u,\varphi\rangle$.
	 Despite this fact, it is not difficult to see that all the integrals
	 appearing in \eqref{eq:defDistributionalSolByParts}
	 are finite. 
	 
	 Indeed, given any
	 $\varphi\in C_0^\infty(\Omega)$, since $g(u),\,h(u)f\in L^1_{\mathrm{loc}}(\Omega)$, we clearly have
	 $$\int_{\Omega}g(u)|\varphi|\,dx<+\infty\quad
	 \text{and}\quad\int_\Omega h(u)f\,|\varphi|\,dx < +\infty.$$
	 Moreover, since $u|_\Omega\in W_0^{1,1}(\Omega)$,  we also have
	 $$\int_\Omega |\nabla u\cdot\nabla \varphi|\,dx 
	 \leq \sup_{\mathbb{R}^n}|\nabla\varphi|\int_{\Omega}|\nabla u|\,dx < +\infty.
	 $$
 Finally, we turn to the \emph{non-local term}  in \eqref{eq:defDistributionalSolByParts}.  To this end
 we first observe that, since  $u|_\Omega\in W_0^{1,1}(\Omega)$
 and since
  $u=0$ a.e.\,on $\mathbb{R}^n\setminus\Omega$  we have
 $$u\in W^{1,1}(\mathbb{R}^n);$$ 
 as a consequence, since Remark \ref{rem:spaceX12}-c) holds also
 for $p = 1$, we obtain $u\in W^{s,1}(\mathbb{R}^n)$.
 From this, we arguing as for the local part, we get \begin{align*}
 	&\iint_{\mathbb{R}^{2n}}\frac{|u(x)-u(y)|\,|\varphi(x)-\varphi(y)|}{|x-y|^{n+2s}}\,dx\,dy \\
 	&\quad \leq \sup_{x\neq y}\frac{|\varphi(x)-\varphi(y)|}{|x-y|^s}\cdot\int_{\mathbb{R}^{2n}}
 	\frac{|u(x)-u(y)|}{|x-y|^{n+s}}\,dx\,dy \\
 	& \quad \leq c\,\|\varphi\|_{C^1(\mathbb{R}^n)}\cdot[u]_{1,s,\mathbb{R}^n}<+\infty.
 \end{align*}

 2)\,\,In the particular case when $u:\mathbb{R}^n\to\mathbb{R}$ is a solution of problem \eqref{eq:Problem} further satisfying $u|_\Omega\in H_0^1(\Omega)$ (notice that the existence of such a solution is ensured, e.g., by Theorem \ref{thm:Existence}), we have $u\in \mathcal{X}^{1,2}(\Omega)$, and identity \eqref{eq:defDistributionalSolByParts} becomes
$$\langle \mathcal{L}u,\varphi\rangle +\int_{\Omega}g(u)\varphi\,dx
= \int_{\Omega}h(u)f\,\varphi\,dx\quad\text{for all $\varphi\in C_0^\infty(\Omega)$}$$ 
 (see the above Remark \ref{rem:LuwhenuX12}). Moreover, we have the following lemma.
 \begin{lemma} \label{lem:LargerTest}
 	Let $u:\mathbb{R}^n\to\mathbb{R}$ be a distributional solution of problem \eqref{eq:Problem} 
 	(in the sense of De\-fi\-nition \ref{def:Distributional_Solution}), further satisfying $u|_\Omega\in H_0^1(\Omega)$. Then,
 	\begin{equation} \label{eq:ByPartsAllargata}
 		\langle \mathcal{L}u,\varphi\rangle +\int_{\Omega}g(u)\varphi\,dx
= \int_{\Omega}h(u)f\,\varphi\,dx\quad\text{for all $\varphi\in \mathcal{X}^{1,2}(\Omega)\cap L^\infty(\Omega)$},
 	\end{equation}
 	that is, identity \eqref{eq:defDistributionalSolByParts} can be extended to all functions $\varphi\in \mathcal{X}^{1,2}(\Omega)\cap L^\infty(\Omega)$.
 \end{lemma}
\begin{proof}
 The proof of such a lemma is totally analogous to that of
 \cite[Theorem 6.1]{Oliva}; however, we present it here in details for the sake of completeness. 
 \medskip
 
 First of all we observe that, by linearity, we can limit ourselves to estalbish
 \eqref{eq:ByPartsAllargata} for \emph{non-ne\-gative} $\varphi$'s. We then
 choose a \emph{non-negative function $\varphi\in \mathcal{X}^{1,2}(\Omega)\cap L^\infty(\Omega)$}, and we let $\{\varphi_k\}_k$ be a sequence 
 in $C_0^\infty(\Omega)$ converging as $k\to+\infty$ to $\varphi$
 \emph{in $\mathcal{X}^{1,2}(\Omega)$}.
 Accordingly, given any $\eta > 0$, we define
 $$v_{\eta,k} = \rho_\eta*(\min\{\varphi_k,\varphi\}),$$
 where $\{\rho_\eta\}_\eta$ is a smooth mollifier in $\mathbb{R}^n$, and $*$ denotes the classical convolution.
 
 Setting $v_k = \min\{\varphi_k,\varphi\}$, it is easy to see that
 \begin{enumerate}
   \item $v_{\eta,k}\in C_0^\infty(\Omega)$ for every $\eta > 0$ and $k\in\mathbb{N}$;
   \vspace{0.1cm}
   
 	\item  $v_{\eta,k}\to v_k$ as $\eta\to 0^+$ in $\mathcal{X}^{1,2}(\Omega)$
 	and weakly\,-\,$*$ in $L^\infty(\Omega)$;
 	\vspace{0.1cm}
 	
 	\item $v_{k}\to \varphi$ as $k\to+\infty$ in $\mathcal{X}^{1,2}(\Omega)$
 	and weakly\,-\,$*$ in $L^\infty(\Omega)$;
 	\vspace{0.1cm}
 	
\item $\mathrm{supp}(v_k)\Subset\Omega$ and $0\leq v_k\leq \varphi$ for every $k\in\mathbb{N}$.
 \end{enumerate}
 Moreover, by possibly choosing a subsequence, we may assume that
 \begin{itemize}
 \item[(5)] $v_k\to \varphi$ a.e.\,in $\Omega$ (hence, in $\mathbb{R}^n$) as $k\to+\infty$.
 \end{itemize}
In particular, owing to (1) above (and  since $u$ solves\eqref{eq:Problem}), we can use  $v_{\eta,k}$ as a test function
in \eqref{eq:defDistributionalSolByParts}, obtaining the following identity:
\begin{equation} \label{eq:toPassLimitetakLemmaOliva}
\langle \LL u,v_{\eta,k}\rangle + \int_{\Omega}g(u)v_{\eta,k} \, dx \\
= \int_{\Omega}h(u)f v_{\eta,k}\, dx.
\end{equation}
We now aim to pass to the limit as $\eta\to 0^+$ and as $k\to+\infty$ in the above \eqref{eq:toPassLimitetakLemmaOliva}.
\medskip

a)\,\,\emph{Limit as $\eta\to 0^+$.}  
As regards the operator part
of \eqref{eq:toPassLimitetakLemmaOliva} it suffices to observe that,  since  we know that $v_{\eta,k}\to v_k$ as $\eta\to 0^+$ \emph{in $\mathcal{X}^{1,2}(\Omega)$}, and since $\mathcal{L}u$ is continuous on $\mathcal{X}^{1,2}(\Omega)$ (recall that $u\in \mathcal{X}^{1,2}(\Omega)$ and see Remark \ref{rem:LuwhenuX12}), we have
$$\langle \mathcal{L}u,v_{\eta,k}\rangle\to \langle \mathcal{L}u,v_{k}\rangle
\quad\text{as $\eta\to 0^+$}.$$
  As regards the remaining two terms, instead, we rely on the weak-$*$ convergence of the sequence $v_{\eta,k}$. Indeed, since $u$ is a solution of \eqref{eq:Problem}, we know that
  $$g(u),\,h(u)f\in L^1_{\mathrm{loc}}(\Omega);$$
as a consequence, since $v_{\eta,k}\to v_k$ weakly\,-\,$*$ in $L^\infty(\Omega)$ (as $\eta\to 0^+$), and since there exists an open set $\mathcal{V} = \mathcal{V}_k\Subset\Omega$ such that
$$\mathrm{supp}(v_{\eta,k})\subseteq\eta+\mathrm{supp}(v_k)\Subset \mathcal{V}_k$$
(provided that $\eta$ is sufficiently small), we get, as $\eta\to 0^+$,
$$\int_{\Omega}g(u)v_{\eta,k} \, dx\to \int_\Omega g(u)v_k\,dx\quad\text{and} 
\quad
\int_{\Omega}h(u)f v_{\eta,k}\, dx\to\int_{\Omega}h(u)f\,v_k\,dx.$$
Gathering all these facts, we can let
$\eta\to 0^+$ in \eqref{eq:toPassLimitetakLemmaOliva}, obtaining
\begin{equation} \label{eq:toPassLimitKOliva}
	\langle \LL u,v_{k}\rangle + \int_{\Omega}g(u)v_{k} \, dx \\
= \int_{\Omega}h(u)f v_{k}\, dx.
\end{equation}

b)\,\,\emph{Limit as $k\to+\infty$.} As for the operator part
of \eqref{eq:toPassLimitKOliva}, thanks to property (3)
and by proceeding exactly as above, we immediately get
$$\langle \LL u,v_{k}\rangle\to \langle \LL u,\varphi\rangle\quad\text{as $k\to+\infty$}.$$ 
Moreover, using property (3) once again, and noting that $\mathrm{supp}(v_k) \subseteq \mathrm{supp}(\varphi) \Subset \Omega$ for all $k \geq 1$ (see property (4)), we can handle the 
\emph{absorption term} on the right-hand side of \eqref{eq:toPassLimitKOliva} exactly as before. Thus, as $k \to +\infty$, we obtain 
 $$\int_{\Omega}g(u)v_{\eta,k} \, dx\to \int_\Omega g(u)\varphi\,dx.$$
 Hence, it remains to analyze the singular term on the left-hand side of \eqref{eq:toPassLimitKOliva}.
 
 To this end we first observe that, owing to \eqref{eq:toPassLimitKOliva}, we get
 \begin{align*}
 	\int_{\Omega}h(u)f\,v_k\,dx & = \int_{\Omega}g(u)v_k\,dx
 	+\langle \LL u,v_{k}\rangle \\
 	& \leq \|g(u)\|_{L^1(\mathcal{V})}\|v_k\|_{L^\infty(\Omega)}
 	+ \|\mathcal{L}u\|_{(\mathcal{X}^{1,2}(\Omega))'}\|v_k\|_{X^{1,2}(\Omega)} 
 	\\
 	& \leq C\qquad\text{for every $k\geq 1$},
 \end{align*}
 where we have used the fact that $v_k\to \varphi$ as $k\to+\infty$ in
 $\mathcal{X}^{1,2}(\Omega)$, and $\mathcal{V}\Subset\Omega$ is a fixed open
 set such that $\mathrm{supp}(v_k)\subseteq\mathcal{V}$ for every $k\geq 1$;
 as a consequence, by using the Fatou Lemma 
 (and since $v_k\to\varphi$ a.e.\,in $\Omega$), we derive that
 \begin{equation} \label{eq:FatouvkOliva}
 	\int_{\Omega}h(u)f\,\varphi\,dx \leq C.
 \end{equation}
 With \eqref{eq:FatouvkOliva} at hand,  we can easily
 complete the proof of the lemma. 
 
 Indeed, since $v_k\to\varphi$ a.e.\,in $\Omega$ as $k\to+\infty$, and since 
 $$0\leq h(u)f\,v_k\leq h(u)f\,\varphi\in L^1(\Omega)$$
 (see property (4), and thanks to \eqref{eq:FatouvkOliva}), we can apply the Dominated Convergence Theorem to the singular term
 in the right-hand side of \eqref{eq:toPassLimitKOliva}, getting
 $$\int_{\Omega}h(u)f v_{k}\, dx\to\int_{\Omega}h(u)f\,\varphi\,dx\quad
 \text{as $k\to+\infty$}.$$
 Gathering all these facts, we can then let $k\to+\infty$ in
 \eqref{eq:toPassLimitKOliva}, obtaing
 $$\langle \mathcal{L}u,\varphi\rangle+\int_\Omega g(u)\varphi\,dx = 
 \int_{\Omega}h(u)f\,\varphi\,dx.
 $$
 This, together with the arbitrariness of $\varphi$, gives the desired \eqref{eq:ByPartsAllargata}.
 \end{proof}
  \medskip  
  
  \noindent\textbf{4)  A Leray-Lions-type Theorem.}
 We now state the announced Leray-Lions-ty\-pe theorem for our mixed operator $\mathcal{L}_p$
 which will be used in the sequel.
  \vspace{0.1cm}
  
  To this end we first recall that, 
  if $\Omega\subseteq\mathbb{R}^n$ is a bounded
  open set, by Remark \ref{rem:LuwhenuX12}
  we know that
  the following map
  $$u\ni \mathcal{X}^{1,p}(\Omega)\mapsto \mathcal{L}_pu\in(\mathcal{X}^{1,p}(\Omega))'$$
  is well-defined and continuous; moreover, it is not difficult to see that
  this map, still denoted by $\LL_p$, is monotone, coercive and hemicontinuous (see e.g. \cite[Section 3]{GarainLindgren}), and therefore it is surjective (see \cite[Corollary 2.2]{Showalter}). Finally, $\mathcal{L}_p$ is pseudomonotone (see \cite[Proposition 2.3]{Showalter}).
\medskip

Hence, we can have the following theorem.  
  \begin{theorem}\label{thm:LerayLions}
  	Let $p \in (1,+\infty)$ and $s\in (0,1)$. Let $a:\Omega \times \mathbb{R}\times \mathbb{R}^{n} \to \mathbb{R}^{n}$ and $F:\Omega \times \mathbb{R}\times \mathbb{R}^{n} \to \mathbb{R}$ be Carath\'{e}odory functions satisfying the following properties:
  	\begin{itemize}
  		\item there exists $\beta>0$ such that $|a(x,\sigma,\xi)| \leq \beta \, (|\sigma|^{p-1}+ |\xi|^{p-1})$;
  		\item there exists $\alpha>0$ such that $a(x,\sigma,\xi) \cdot \xi \geq \alpha |\xi|^p$, for all $\xi \in \mathbb{R}^{n}$;
  		\item $(a(x,\sigma,\xi)-a(x,\sigma,\eta) \cdot (\xi -\eta) >0$ for every $\xi \neq \eta$;
  		\item there exists $f\in L^{p'}(\Omega)$ such that $|F(x,\sigma,\eta)|\leq f(x)$.
  	\end{itemize}
  	Then there exists a solution $u\in \mathcal{X}^{1,p}(\Omega)$ of  
  \begin{equation}\tag{mLL}\label{eq:MixedLerayLions}
  	\left\{ \begin{array}{rl}
  		-\mathrm{div}(a(x,u,\nabla u)) + (-\Delta)_{p}^{s} u = F(x,u,\nabla u) & \textrm{in } \Omega,\\
  		u=0 & \textrm{in } \mathbb{R}^{n}\setminus \Omega.
  	\end{array}\right.
  \end{equation}
  \end{theorem}
  	In absence of the nonlocal part $(-\Delta)_p^s$, Theorem \ref{thm:LerayLions} has been proved in \cite{LeLi} as an application of the surjectivity theorem. The crucial part of this proof consists in showing that the operator 
  	\begin{equation}
  		W_0^{1,p}(\Omega)\ni u\mapsto A(u)-\mathrm{div}(a(x,u,\nabla u)) - F(x,u,\nabla u)\in (W_0^{1,p}(\Omega))'
  	\end{equation}
  	\noindent is pseudomonotone and coercive. 
  	On the other hand,
  	since we know that the opera\-tor
  	 $u\mapsto (-\Delta)_p^s u$ is 
  	continuous from $\mathcal{X}^{1,p}(\Omega)$ to its dual,
  	see \eqref{eq:DeltapsonX1pbuono},
 and since it is not difficult to recognize that
  	\begin{align*}
  	(\ast)&\,\,\langle (-\Delta)^s_p u-(-\Delta)^s_p v, u-v\rangle \geq 0\\
  	(\ast)&\,\,\langle (-\Delta)^s_p u, u\rangle = [u]_{s,p,\mathbb{R}^n}\geq 0 	
  	\end{align*}
  	for every $u,v\in\mathcal{X}^{1,p}(\Omega)$ (see
  	\eqref{eq:explicitDeltapsX1p} and recall \eqref{eq:defX12explicit}),
  	 the mixed operator 
  	 $$A' = A+(-\Delta)^s_p$$ 
  	 we are  considering satisfies the assumptions of the surjectivity theorem as well.
  \vspace{0.1cm}
  
  	We explicitly point out that the argument outlined before shows that one is al\-lowed to perturb the original {\it Leray-Lions operator} with a more general nonlocal operator with $p$-growth provided it is pseudomonotone and non-negative when defined on the appropriate reflexive Banach space.
\bigskip

\noindent \textbf{4) Rearrangements.}
Finally, we end this section by recalling some basic notions and results about rearrangements that will be needed in Section \ref{sec:talenti}. For an exhaustive treatment on this topic see e.g. \cite{Talenti,Talenti2}.
\vspace{0.1cm}

Let $\Omega \subset \mathbb{R}^{n}$ be an open and bounded set and let $u: \Omega \rightarrow \mathbb{R}$ be a measurable function. We denote by $\mu_{u}$ the distribution function of $u$, that is
$$
\mu_{u}(t):=|\{x \in \Omega:|u(x)|>t\}|, \quad\text{for } t \geq 0
$$
and by $u^{*}$ its \emph{decreasing rearrangement} as the generalized inverse of $\mu_u$, that is
$$
u^{\ast}(s):=\sup \left\{t\geq 0: \mu_{u}(t)>s\right\}, \quad \text{for }s \in(0,|\Omega|).
$$
The radially symmetric, decreasing rearrangement of $u$, also known as the \emph{Schwarz decreasing rearrangement} of $u$, is hence defined as 
$$u^{\sharp}(x)=u^\ast(\omega_n|x|^n)\qquad x \in \Omega^\sharp,
$$
where $\omega_n$ is the measure of the unitary ball in $\R^n$, and $\Omega^\sharp$ is the ball (centered at the origin) having the same measure as $\Omega$.
From the definitions given above we can easily deduce that $u$, $u^\ast$ and $u^\sharp$ are equi-distributed, that is 
$$
\mu_u=\mu_{u^\ast}=\mu_{u^\sharp}.
$$
Moreover, from the definition, we can easily deduce the following properties that will be useful in the sequel:
\begin{itemize}
	\item[i)]if $|u(x)| \leq|v(x)|$ for a.e. $x \in \Omega$, then $u^{*}(s) \leq v^{*}(s), s \in(0,|\Omega|)$;
	\item[ii)]for any $c \in \mathbb{R},(u+c)^{*}(s)=u^{*}(s)+c, s \in(0,|\Omega|)$;
	\item[iii)] if $u \in L^{p}(\Omega), 1 \leq p \leq \infty$, then $u^{*} \in L^{p}(0,|\Omega|)$ and $\|u\|_{L^{p}(\Omega)}=\left\|u^{*}\right\|_{L^{p}(0,|\Omega|)}$;
	\item[iv)] if $u, v \in L^{p}(\Omega), 1 \leq p \leq \infty$, then
	$$
	\left\|u^{*}-v^{*}\right\|_{L^{p}(0,|\Omega|)} \leq\|u-v\|_{L^{p}(\Omega)} .
	$$
\end{itemize}

Moreover, the following inequality, known as Hardy-Littlewood inequality, holds true
\begin{equation*}
	\int_{\Omega}|u(x) v(x)| \,d x \leq \int_{0}^{|\Omega|} u^{*}(s) v^{*}(s) \,d s. 
\end{equation*}

  \section{Existence and uniqueness via approximating problems}\label{sec:Existence_and_uniqueness}
  In this section we provide the proofs of Theorems \ref{thm:Existence} and \ref{thm:Uniqueness}.
  \subsection{Existence}
For every $n,k\in \mathbb{N}$, we use the truncation functions $G_{k}(\cdot)$ 
(defined in \eqref{eq:def_Gk}) and $T_{n}(\cdot)$ (defined in \eqref{eq:def_Tk}) to 
set 
 \begin{equation} \label{eq:fngnhnGeneral}
 \begin{gathered}
 f_n:= T_{n}(f),
\qquad
 g_{k}(s):= \left\{ \begin{array}{rl}
 T_{k}(g(s)) & s\geq 0,\\
 0 & s< 0,
 \end{array}\right.
\\[0.1cm]
 h_{n}(s):= \left\{ \begin{array}{rl}
 T_{n}(h(s)) & s\geq 0,\\
 h(0) & s< 0.
 \end{array}\right.
 \end{gathered}
 \end{equation}
We then consider the following family of approximating problems
\begin{equation}\label{eq:Approx_Problem}
\left\{ \begin{array}{rl}
\mathcal{L}u_{n,k} + g_{k}(u_{n,k}) = h_n(u_{n,k})f_n & \textrm{in } \Omega,\\
u_{n,k} \geq 0 & \textrm{in } \Omega,\\
u_{n,k} = 0 & \textrm{in } \mathbb{R}^{n}\setminus \Omega.\end{array}\right.
\end{equation}  
The weak formulation of problem \eqref{eq:Approx_Problem} reads as follows:
\begin{equation}\label{eq:Weak_Approx}
\mathcal{B}(u_{n,k}, \varphi) + \int_{\Omega}g_{k}(u_{n,k})\varphi = \int_{\Omega}h_{n}(u_{n,k})f_{n}\varphi, \quad  \varphi \in \mathcal{X}^{1,2}(\Omega).
\end{equation}  
 We explicitly observe that we allowed to use $\mathcal{X}^{1,2}(\Omega)$ 
 as test\,-\,function space in \eqref{eq:Weak_Approx} since, by definition, $g_k,h_n$ and $f_n$ are globally bounded.

\begin{lemma}\label{lem:First_Approx}
The approximating problem \eqref{eq:Approx_Problem} admits a weak solution $u_{n,k}\in \mathcal{X}^{1,2}(\Omega)$ such that
\begin{itemize}
\item[i)] $u_{n,k}\in L^{\infty}(\Omega)$;
\item[ii)] $u_{n,k}$ is bounded in $L^{\infty}(\Omega)$ uniformly in $k\in\mathbb{N}$;
\item[iii)] $u_{n,k}$ is bounded in $\mathcal{X}^{1,2}(\Omega)$ uniformly in $k\in\mathbb{N}$.
\end{itemize}
\end{lemma}  
  \begin{proof}
  First, we notice that Theorem \ref{thm:LerayLions} provides the existence of a
  weak solution $u_{n,k}\in \mathcal{X}^{1,2}(\Omega)$ for every $k,n\in \mathbb{N}$ fixed. Moreover, exploiting \cite{BDVV}, one immediately obtains that $u_{n,k}\in L^{\infty}(\Omega)$. 
  \medskip
  
  Concerning i), recalling Remark \ref{rem:spaceX12}-1), we can test the weak formulation \eqref{eq:Weak_Approx} with $\varphi = u_{n,k}^{-}$, getting
\begin{equation}\label{eq:Testata_con_unk^-}
0\geq \mathcal{B}(u_{n,k}, u_{n,k}^{-}) + \int_{\Omega}g_{k}(u_{n,k})u_{n,k}^{-}\,dx = \int_{\Omega}h_{n}(u_{n,k})f_{n}u_{n,k}^{-}\,dx \geq 0,
\end{equation} 
  \noindent where the last inequality is a consequence of the nonnegativity of $h_{n}(u_{n,k})f_{n}u_{n,k}^{-}$. On the other hand, concerning the first inequality, it holds that
  \begin{equation*}
  0 = g_{k}(u_{n,k}(x))u_{n,k}^{-}(x) = \left\{ \begin{array}{rl}
  g_{k}(u_{n,k})\cdot 0, & x \in \{u_{n,k}\geq 0\}\\
  0 \cdot (-u_{n,k}),  & x \in \{u_{n,k}<0\},
  \end{array}\right.
  \end{equation*}
  \noindent and  
 \begin{equation*}
 \begin{aligned}
 & \mathcal{B}(u_{n,k}, u_{n,k}^{-}) \\
 & \qquad = -\int_{\Omega}|\nabla u_{n,k}^{-}|^2 \, dx + \iint_{\mathbb{R}^{2n}}\dfrac{(u_{n,k}(x)-u_{n,k}(y))(u_{n,k}^{-}(x)-u_{n,k}^{-}(y))}{|x-y|^{n+2s}}\,dxdy\\
 &\qquad\leq -\mathcal{B}(u_{n,k}^{-}, u_{n,k}^{-})\leq 0,
 \end{aligned}
\end{equation*}  
\noindent where we used that 
\begin{align*}
(LHS)&:= (u_{n,k}(x)-u_{n,k}(y))(u_{n,k}^{-}(x)-u_{n,k}^{-}(y)) \\
& \leq -(u_{n,k}^{-}(x)-u_{n,k}^{-}(y))^2 =: (RHS),	
\end{align*}
\noindent for every $x,y\in \mathbb{R}^{n}$. Indeed:
\begin{itemize}
\item if $x,y \in \{u_{n,k}\geq 0\}$, then $(LHS) = 0 =(RHS)$\,
\item if $x \in \{u_{n,k}\geq 0\}$ and $y \in \{u_{n,k}<0\}$, then
$$(LHS) = u_{n,k}(x)u_{n,k}(y) - u_{n,k}^{2}(y) \leq - u_{n,k}^{2}(y) =(RHS);$$
\item if $x \in \{u_{n,k}< 0\}$ and $y \in \{u_{n,k}\geq 0\}$, then
$$(LHS) = - u_{n,k}^{2}(x) + u_{n,k}(x)u_{n,k}(y)  \leq - u_{n,k}^{2}(x) =(RHS);$$
\item if $x,y \in \{u_{n,k}< 0\}$, then 
$$(LHS) = -(u_{n,k}(x)-u_{n,k}(y))^2 = (RHS).$$
\end{itemize}
  Therefore, going back to \eqref{eq:Testata_con_unk^-}, we find 
  $$0=\mathcal{B}(u_{n,k}^{-}, u_{n,k}^{-}) \geq \int_{\Omega}|\nabla u_{n,k}^{-}|^2 \, dx \geq C \, \|u_{n,k}^{-}\|^{2}_{L^{2}(\Omega)},$$
  \noindent and this implies that $u_{n,k}^{-} = 0$ in $\Omega$ and thus $u_{n,k}\geq 0$ a.e.\,in $\Omega$.\\
  
Regarding ii), for any given $t>0$ we take
\begin{equation}
G_{t}(u_{n,k}) \stackrel{\eqref{eq:def_Gk}}{=} (|u_{n,k}|-t)^{+}\textrm{sign}(u_{n,k}) \stackrel{(u_{n,k}\geq 0)}{=} (u_{n,k}-t)^{+} \in \mathcal{X}^{1,2}(\Omega),
\end{equation}  
  \noindent and we use as a test function in \eqref{eq:Weak_Approx}, finding
\begin{equation}\label{eq:Testata_con_G_t}
\mathcal{B}(u_{n,k}, G_{t}(u_{n,k})) + \int_{\Omega}g_{k}(u_{n,k})G_{t}(u_{n,k})\,dx = \int_{\Omega}h_{n}(u_{n,k})f_{n}G_{t}(u_{n,k})\,dx.
\end{equation} 
Now, clearly 
$$g_{k}(u_{n,k})G_{t}(u_{n,k}) \geq 0,$$
\noindent and
\begin{equation}
\begin{aligned}
&\mathcal{B}(u_{n,k}, G_{t}(u_{n,k}))  = \int_{\Omega}\nabla u_{n,k}\cdot \nabla G_{t}(u_{n,k})\, dx \\
&\qquad  + \iint_{\mathbb{R}^{2n}}\dfrac{(u_{n,k}(x)-u_{n,k}(y))(G_{t}(u_{n,k}(x))-G_{t}(u_{n,k}(x)))}{|x-y|^{n+2s}}\, dxdy \\
& \geq \int_{\Omega}|\nabla G_{t}(u_{n,k})|^{2}\, dx,
\end{aligned}
\end{equation}  
\noindent where the inequality for the nonlocal part follows from the fact that
$$(u_{n,k}(x)-u_{n,k}(y))((u_{n,k}-t)^{+}(x)-(u_{n,k}-t)^{+}(y)) \geq 0\quad
\text{for a.e.\,$x,y\in \mathbb{R}^{n}$},$$
 since the map $z\mapsto (z-t)^+$ is non-decreasing.

 Going back to \eqref{eq:Testata_con_G_t}, we finally have
 \begin{equation}
 \begin{aligned}
 \int_{\Omega}|\nabla G_{t}(u_{n,k})|^{2}\, dx &\leq \int_{\Omega}h_{n}(u_{n,k})f_{n}G_{t}(u_{n,k})\, dx\\
 &\leq \sup_{s\in [t,+\infty]}h(s) \, \int_{\Omega}f_{n}G_{t}(u_{n,k})\, dx.
 \end{aligned}
\end{equation}  
  \noindent The latter estimate is the starting point to run the nowadays classical procedure due to Stampacchia (see \cite{Stampacchia}). In the present case, this implies the existence of a positive constant $C>0$ independent of $k$ such that $\|u_{n,k}\|_{L^{\infty}(\Omega)}\leq C$. 
  
Concluding with iii), we can use the solution $u_{n,k}$ itself as a test function in \eqref{eq:Weak_Approx}, getting
\begin{equation}\label{eq:Testata_con_u_nk}
\begin{aligned}
\mathcal{B}(u_{n,k},u_{n,k}) & \leq \mathcal{B}(u_{n,k},u_{n,k}) + \int_{\Omega}g_{k}(u_{n,k})u_{n,k}\,dx \\
&=\int_{\Omega}h_{n}(u_{n,k})f_{n}u_{n,k}\,dx \leq C,
\end{aligned} 
\end{equation}  
  \noindent for some positive constant $C>0$, depending on $n$, $\Omega$ and $h$, but independent of $k\in \mathbb{N}$ (here we are also using assertion ii)).
  The latter \eqref{eq:Testata_con_u_nk} shows that $\rho(u_{n,k})$ is bounded uniformly on $k$ and hence that $u_{n,k}$ is bounded uniformly on $k$ in $\mathcal{X}^{1,2}(\Omega)$. 

This closes the proof.
  \end{proof}
  
 Thanks to Lemma \ref{lem:First_Approx}, we can easily 
 pass to the limit as $k\to +\infty$ obtaining, up to subsequences, a limit function $u_n$ which is a weak solution of
\begin{equation}\label{eq:Post_limit_in_k}
\left\{ \begin{array}{rl}
\mathcal{L}u_n + g(u_n) = h_{n}(u_n)f_n &\textrm{ in } \Omega,\\
u_n \geq 0 & \textrm{ in } \Omega,\\
u_n = 0 & \textrm{ in } \mathbb{R}^{n}\setminus \Omega.
\end{array}\right.
\end{equation}
Indeed, using assertion iii) in the cited Lemma \ref{lem:First_Approx},
jointly with Remark \ref{rem:spaceX12}-4),
we can find a function $u_n\in \mathcal{X}^{1,2}(\Omega)$ such that
(up to a subsequence)
\begin{itemize}
    \item[a)] $u_{n,k}\to u_n$ weakly in $\mathcal{X}^{1,2}(\Omega)$;
    \item[b)] $u_{n,k}\to u_n$ \emph{strongly} in $L^m(\Omega)$ for every $1\leq m < 2^\ast$;
    \item[c)] $u_{n,k}\to u_n$ pointwise a.e.\,in $\Omega$.
   \end{itemize}
Now, given any $\varphi\in \mathcal{X}^{1,2}(\Omega)$, by property a) we have
$$\mathcal{B}(u_{n,k},\varphi)\to \mathcal{B}(u_n,\varphi)\quad\text{as $k\to+\infty$}.$$
Moreover, by assertion ii) in Lemma \ref{lem:First_Approx} and
assumption (H)$_g$, we have
$$0\leq g_k(u_{n,k})\leq g(u_{n,k}) \leq \max_{[0,M]}g\quad\text{for every $k\geq 1$},$$
where $M = M_n > 0$ is such that $\|u_{n,k}\|_{L^\infty(\Omega)}\leq M$ for every $k\geq 1$; this, together with the above property c), allows us to use
the Dominated Convergence Theorem, from which we easily derive that
$$\int_{\Omega}g_k(u_{n,k})\varphi\,dx\to \int_{\Omega}g(u_n)\varphi\,dx\quad
\text{as $k\to+\infty$}.$$
Finally, again by Lemma \ref{lem:First_Approx}-ii) and the definition of $T_n$, we have
$$0\leq h_n(u_{n,k})f_n = T_n\big(h(u_{n,k})\big)\,T_n(f)\leq n^2\quad\text{for every $k\geq 1$};$$
hence, using property c) and again the Dominated Convergence Theorem, we get
$$\int_{\Omega}h_n(u_{n,k})f_n\varphi\,dx\to \int_{\Omega}h_n(u_{n})f_n\varphi\,dx\quad
\text{as $k\to+\infty$}.$$
Gathering all these facts, from \eqref{eq:Weak_Approx} we then infer that
\begin{equation}\label{eq:Weak_Post_limit_in_k}
\mathcal{B}(u_n,\varphi) + \int_{\Omega}g(u_n)\varphi\,dx = \int_{\Omega}h_n(u_n)f_n \varphi\,dx, \quad \varphi \in \mathcal{X}^{1,2}(\Omega).
\end{equation}
We explicitly notice that also in the above \eqref{eq:Weak_Post_limit_in_k}
we can use $\mathcal{X}^{1,2}(\Omega)$ as a test-fun-ction
space (since $g(u_n)\in L^\infty(\Omega)$, with a norm depending on $n$).
\vspace{0.1cm}

The following Lemma provides the necessary apriori estimates to further pass to the limit (as $n\to +\infty$) in \eqref{eq:Post_limit_in_k}. This an almost verbatim generalization of \cite[Lemma 4.1]{Oliva} to the local-nonlocal case.  
  
\begin{lemma} \label{lem:propun}
Let $u_n$ be a weak solution to \eqref{eq:Post_limit_in_k}.
Let $h$ satisfy {\rm (h1)} with $\gamma \leq 1$ and {\rm (h2)}. If one of the following holds
\begin{itemize}
\item[i)] $\theta \geq 1$ and $f\in L^{1}(\Omega)$;
\item[ii)] $\theta <1$, $f\in L^{m}(\Omega)$ with $m>1$, and $g$ satisfies {\rm (g1)},
\end{itemize}
\noindent then there exists a positive constant $C>0$  independent of $n$ such that
\begin{equation}\label{eq:Apriori_estimates}
\rho(u_n)^2 + \|g(u_n)u_n\|_{L^{1}(\Omega)} \leq C.
\end{equation}
Moreover, there exists a positive constant (still denoted by $C$) independent of $n$ such that
\begin{equation}\label{eq:Apriori_estimates_2}
\int_{\Omega}h_n(u_n)f_n \varphi\,dx \leq C + \rho(\varphi)^2 + C\|\varphi\|_{L^{\infty}(\Omega)},
\end{equation}
\noindent for every $\varphi \in \mathcal{X}^{1,2}(\Omega)\cap L^{\infty}(\Omega)$.
\end{lemma}  
  \begin{proof}
  Firstly we test \eqref{eq:Post_limit_in_k} with the solution $u_n$ itself, finding
  \begin{equation}\label{eq:testing_with_u_n}
  \begin{aligned}
  &\mathcal{B}(u_n,u_n) + \int_{\Omega}g(u_n)u_n\,dx = \int_{\Omega}h_n(u_n)f_n u_n\,dx\\ &\leq \underline{C} \int_{\{u_n <\underline{s}\}}f_n u_n^{1-\gamma}\,dx  + \int_{\{\underline{s}\leq u_n \leq \overline{s}\}}h_n(u_n)f_n u_n\,dx + \overline{C}\int_{\{u_n>\overline{s}\}}f_n u_n^{1-\theta}\,dx,
  \end{aligned}
  \end{equation}
  \noindent where $\underline{C}$, $\gamma$ and $\underline{s}$ come from {\rm (h1)} while $\overline{C}$, $\theta$ and $\overline{s}$ from {\rm (h2)}.
  Now, in both cases i) and ii), we have that 
  \begin{equation}\label{eq:Holding_for_i_and_ii}
  \begin{aligned}
  \underline{C} \int_{\{u_n <\underline{s}\}}f_n u_n^{1-\gamma}\,dx  & + \int_{\{\underline{s}\leq u_n \leq \overline{s}\}}h_n(u_n)f_n u_n\,dx \\
&\leq \left( \underline{C}\, \underline{s}^{1-\gamma} + \max_{s \in [\underline{s},\overline{s}]}(h(s)s)\right)\|f\|_{L^{1}(\Omega)}
\end{aligned}
  \end{equation}

In case {\rm i)} holds, we can simply estimate the  very last term in the right hand side of \eqref{eq:testing_with_u_n} as follows
\begin{equation}\label{eq:Third_term_i}
\overline{C}\int_{\{u_n>\overline{s}\}}f_n u_n^{1-\theta}\,dx
\leq \overline{C}\overline{s}^{1-\theta}\, \|f\|_{L^{1}(\Omega)}.
\end{equation}  
Combining \eqref{eq:testing_with_u_n} with \eqref{eq:Holding_for_i_and_ii} and \eqref{eq:Third_term_i} we prove \eqref{eq:Apriori_estimates} in case {\rm i)} holds for 
\begin{equation*}
C:= \left( \underline{C}\, \underline{s}^{1-\gamma} + \max_{s \in [\underline{s},\overline{s}]}(h(s)s) + \overline{C}\overline{s}^{1-\theta}\right)\|f\|_{L^{1}(\Omega)},
\end{equation*}
\noindent which is independent of $n$.

  Let us now consider the case in which {\rm ii)} is in force. Starting from \eqref{eq:testing_with_u_n}, we can use Young's inequality (with exponents $m$ and $\tfrac{m}{m-1}$) on the very last term finding
  \begin{equation}\label{eq:Third_term_ii}
\overline{C}\int_{\{u_n>\overline{s}\}}f_n u_n^{1-\theta}\,dx
  \leq C_{\varepsilon} \|f\|^{m}_{L^{m}(\Omega)} + \varepsilon \int_{\{u_n > \overline{s}\}}u_{n}^{\tfrac{(1-\theta)m}{m-1}}\,dx,
  \end{equation}
  \noindent for every $\varepsilon>0$ and for some positive constant $C_\varepsilon>0$ depending on $\varepsilon$. From now on, we can assume without loss of generality, that $s_1 \leq \overline{s}$.
  Now, if $q= \tfrac{1-m\theta}{m-1}$ in {\rm (g1)}, it follows that
  $q+1 =\tfrac{(1-\theta)m}{m-1}$ and therefore
  \begin{equation}\label{eq:Choosing_varepsilon}
  \varepsilon \int_{\{ u_n>\overline{s}\geq s_1\}}u_n^{q+1}\,dx \leq \dfrac{\varepsilon}{\nu}\int_{\Omega}g(u_n)u_n\,dx,
  \end{equation}
  \noindent which can reabsorbed for $\varepsilon>0$ small enough (e.g. $\varepsilon = \tfrac{\nu}{2}$). Summarizing, in case {\rm ii)} with $q= \tfrac{1-m\theta}{m-1}$, combining \eqref{eq:testing_with_u_n}, \eqref{eq:Holding_for_i_and_ii}, \eqref{eq:Third_term_ii} and \eqref{eq:Choosing_varepsilon}, we get
  \begin{equation*}
  \begin{aligned}
    \rho(u_n)^2 &+ \dfrac{1}{2}\int_{\Omega}g(u_n)u_n\,dx \\
   & \leq C:=\left( \underline{C}\underline{s}^{1-\gamma} + \max_{s \in [\underline{s},\overline{s}]}(h(s)s) + \overline{C}\overline{s}^{1-\theta}\right)\|f\|_{L^{1}(\Omega)} + C_{\varepsilon} \|f\|^{m}_{L^{m}(\Omega)},
   \end{aligned}
  \end{equation*}
  \noindent which is independent of $n$ and the latter in turn gives
  \begin{equation*}
  \rho(u_n)^2 + \dfrac{1}{2}\int_{\Omega}g(u_n)u_n\,dx \leq 3C.
  \end{equation*}
  The case {\rm ii)} with $q >\tfrac{1-m\theta}{m-1}$ can be treated similarly applying once again Young's inequality with exponents $\left(\tfrac{(q+1)(m-1)}{(1-\theta)m}, \tfrac{(q+1)(m-1)}{(q+1)(m-1)-(1-\theta)m}\right)$ on the last term on the right hand side of \eqref{eq:Third_term_ii}. This choice of exponents, combined with {\rm (g1)}, gives
  \begin{equation}
  \begin{aligned}
  \varepsilon & \int_{\{u_n > \overline{s}\}}u_{n}^{\tfrac{(1-\theta)m}{m-1}}\,dx \\
  &\leq \dfrac{\varepsilon (1-\theta)m}{(q+1)(m-1)} \int_{\{u_n>\overline{s}\geq s_1\}}u_{n}^{q+1}\,dx + \dfrac{\varepsilon[(q+1)(m-1)-(1-\theta)m]}{(q+1)(m-1)}|\Omega|\\
  &\leq \dfrac{\varepsilon (1-\theta)m}{\nu (q+1)(m-1)} \int_{\Omega}g(u_{n})u_n\,dx + \dfrac{\varepsilon[(q+1)(m-1)-(1-\theta)m]}{(q+1)(m-1)}|\Omega|.
  \end{aligned}
  \end{equation}
  \noindent Arguing as before, we can complete the proof of \eqref{eq:Apriori_estimates} in case {\rm ii)} with $q >\tfrac{1-m\theta}{m-1}$. 
 \medskip
 
  Once \eqref{eq:Apriori_estimates} is proved, we are allowed to test \eqref{eq:Post_limit_in_k} with $\varphi \in \mathcal{X}^{1,2}(\Omega)\cap L^{\infty}(\Omega)$  finding
  \begin{equation*}
  \begin{aligned}
  \int_{\Omega}h_n(u_n)f_n \varphi\,dx &= \mathcal{B}(u_n,\varphi) +\int_{\Omega}g(u_n)\varphi\,dx \\
  &\leq \dfrac{\rho(u_n)^2}{2} + \dfrac{\rho(\varphi)^2}{2}+ \|g(u_n)\|_{L^{1}(\Omega)} \|\varphi\|_{L^{\infty}(\Omega)}.
    \end{aligned}
  \end{equation*}
    Combining the latter with \eqref{eq:Apriori_estimates} we get \eqref{eq:Apriori_estimates_2}. This closes the proof. 
  \end{proof}

We are now ready to prove the existence result.
\begin{proof}[Proof of Theorem \ref{thm:Existence}]
We consider a weak solution $u_n\in\mathcal{X}^{1,2}(\Omega)$ to problem \eqref{eq:Post_limit_in_k}.
 By \eqref{eq:Apriori_estimates}, we have that $\rho(u_n)$ is uniformly bounded (i.e., $u_n$ is uniformly bounded in $\mathcal{X}^{1,2}(\Omega)$); hence, by Remark \ref{rem:spaceX12}\,-\,4) there exist a (unique)
 function $u \in \mathcal{X}^{1,2}(\Omega)$ such that
 (by choosing a sub-sequence if needed, and as $n\to+\infty$)
\begin{itemize}
\item[a)] $u_n \to u$ weakly in $\mathcal{X}^{1,2}(\Omega)$;
\item[b)] $u_n \to u$  strongly in $L^{r}(\Omega)$ for every $1\leq r <2^{\ast}$;
\item[c)] $u_n\to u$ 
 pointwise a.e.\,in $\Omega$ (hence, in $\R^n$).
\end{itemize} 
Moreover, by passing to the limit as $n \to +\infty$ in \eqref{eq:Apriori_estimates} with the aid
of the Fatou Lemma, we derive that
\begin{equation} \label{eq:guuL1}
 g(u)u\in L^{1}(\Omega),
 \end{equation}
as requested. We explicitly stress that, owing to 
\eqref{eq:guuL1} (and recalling that $g$ is non-negative
and continuous on $[0,+\infty)$), we also derive that 
$$g(u)\in L^1(\Omega).$$
We now fix a test function $\varphi\in \mathcal{X}^{1,2}(\Omega)\cap L^\infty(\Omega),
\varphi\geq 0$ in $\Omega$,
and we turn to show that \emph{it is possible to pass to the limit as $n\to+\infty$ in \eqref{eq:Weak_Post_limit_in_k}}, thus obtaining
\begin{equation}  \label{eq:finalGoalWeaku}
 \mathcal{B}(u,\varphi)+\int_{\Omega}g(u)\varphi\,dx = \int_{\Omega} h(u)f\varphi\,dx.
 \end{equation}
Due the arbitrariness of $\varphi$, this proves that $u$ is a weak solution of 
\eqref{eq:Problem}. To ease the readability,
we study separately the three terms appearing in the cited \eqref{eq:Weak_Post_limit_in_k}.
\medskip

\noindent-\,\,\emph{The operator part.}
Thanks to {\rm b)}, we immediately have that 
\begin{equation} \label{eq:operatorTerm}
\mathcal{B}(u_n,\varphi) \to \mathcal{B}(u,\varphi) \quad \textrm{as } n \to +\infty.
\end{equation}
-\,\,\emph{The absorption term}.
Regarding the absorption term, our goal is to show that
$$\int_{\Omega}g(u_n)\varphi\,dx$$
\noindent is equi-integrable (with respect to $n$). To this aim, we first test \eqref{eq:Post_limit_in_k} with the function $S_{\eta,k}(u_n)$ (defined in \eqref{eq:def_S_delta_k} with $\eta, k>0$ being fixed), finding
\begin{equation}\label{eq:Test_con_S_eta_k}
	\mathcal{B}(u_n,S_{\eta,k}(u_n)) + \int_{\Omega}g(u_n)S_{\eta,k}(u_n)\,dx = \int_{\Omega}h_n(u_n)f_n \, S_{\eta,k}(u_n)\,dx.
\end{equation}
Now,
\begin{equation}\label{eq:Test_S_eta_k_locale}
	\int_{\Omega}\nabla u_n \cdot \nabla S_{\eta,k}(u_n)\,dx = \int_{\Omega}|\nabla u_n|^2 S'_{\eta,k}(u_n)\,dx,
\end{equation}
\noindent while we claim that
\begin{equation}\label{eq:Test_S_eta_k_nonlocale}
	\iint_{\mathbb{R}^{2n}}\dfrac{(u_n(x)-u_n(y))(S_{\eta,k}(u_n(x))-S_{\eta,k}(u_n(y)))}{|x-y|^{n+2s}}\,dx\,dy \geq 0.
\end{equation}
Indeed, since the map $z\mapsto S_{\eta,k}(z)$ is non-dec\-reasing, we have
\begin{equation*}
	(u_n(x)-u_n(y))(S_{\eta,k}(u_n(x))-S_{\eta,k}(u_n(y))) \geq 0\quad\text{for a.e.\,$x,y\in\mathbb{R}^n$}.
\end{equation*}
Combining \eqref{eq:Test_con_S_eta_k} with \eqref{eq:Test_S_eta_k_locale} and \eqref{eq:Test_S_eta_k_nonlocale}, we further get
\begin{equation*}
	\int_{\Omega}|\nabla u_n|^2 S'_{\eta,k}(u_n)\,dx + \int_{\Omega}g(u_n)S_{\eta,k}(u_n)\,dx \leq (\sup_{s\in [k,+\infty)}h(s))\int_{\Omega}f_n S_{\eta,k}(u_n)\,dx.
\end{equation*}
From now on we can follow verbatim \cite[Proof of Theorem 3.4]{Oliva}, giving
\begin{equation} \label{eq:AbsorTerm}
\lim_{n\to+\infty}\int_\Omega g(u_n)\varphi\,dx = \int_{\Omega} g(u)\varphi\,dx.
\end{equation}
\noindent-\,\,\emph{The singular term.}
Finally, we turn to study the last term of \eqref{eq:Weak_Post_limit_in_k}, that is,
$$\int_{\Omega}h_n(u_n)f_n \varphi\,dx$$
To this end we preliminary notice that, if $h(0)<+\infty$, an application
of the Dominated Convergence Theorem (based on \eqref{eq:assh2}) easily shows that
$$\lim_{n\to+\infty}\int_{\Omega}h_n(u_n)f_n \varphi\,dx = \int_{\Omega}h(u)f \varphi\,dx,$$
and the proof is complete in this case; hence, we assume that
\begin{equation} \label{eq:assumptionSingular}
 h(0) = +\infty,
\end{equation}
and we closely follow the argument in the proof
of \cite[Theorem 3.4]{Oliva}. 

First of all we first observe that, 
by passing to the limit as $n\to+\infty$ 
in \eqref{eq:Apriori_estimates_2} with the aid of the
Fatou Lemma, we get
\begin{equation} \label{eq:boundhufvarphi}
 \int_{\Omega} h(u)f\varphi\,dx \leq C,
\end{equation}
for some $C > 0$ independent of $n$. In particular, by
\eqref{eq:assumptionSingular}\,-\,\eqref{eq:boundhufvarphi},
 we have
 \begin{equation} \label{eq:inclusionufzero}
  \text{$\{u = 0\}\subseteq \{f = 0\}$ (up to a set of zero Lebesgue measure)}.
  \end{equation}
 We then fix
 $\delta > 0$ in such a way that $|\{u = \delta\}| = 0$
(notice that this is certainly possible, since $u\in \mathcal{X}^{1,2}(\Omega)\subseteq L^2(\Omega)$ ensures that $\mathcal{E} = \{\delta:\,|\{u = \delta\}| > 0\}\subseteq\mathbb{R}$ is at most countable, see \cite[Lemma 2.8]{MonticelliRW}), and we write
\begin{equation}\label{eq:splitting_the_third}
\begin{split}
\int_{\Omega}h_n(u_n)f_n \varphi\,dx & = \int_{\{u_n > \delta\}} h_n(u_n)f_n \varphi\,dx 
+ \int_{\{u_n \leq \delta\}} h_n(u_n)f_n \varphi\,dx\\
& = A_{n,\delta}+B_{n,\delta}.
\end{split}
\end{equation}
Taking into account \eqref{eq:splitting_the_third}, we now turn to show that it is possible
 to pass to the limit \emph{as $n\to+\infty$ and as $\delta\to 0^+$} in both terms
$A_{n,\delta},B_{n,\delta}$.
\medskip

-\,\,\emph{The term $A_{n,\delta}$}. As regards this term,
we can proceed exactly as in the proof
of \cite[Theorem 3.4]{Oliva} (where the same term is considered,
and since the \emph{operator term} does not play any role): owing to the properties of $h$
(see, in particular, property \eqref{eq:assh2}),
and taking into account \eqref{eq:boundhufvarphi}, we derive that
\begin{equation} \label{eq:limAndelta}
 \lim_{\delta\to 0^+}\lim_{n\to+\infty}A_{n,\delta} =
\int_\Omega h(u)f\varphi\,dx.
\end{equation}

-\,\,\emph{The term $B_{n,\delta}$}. As regards this term, instead, we claim that
\begin{equation} \label{eq:limBndelta}
 \lim_{\delta\to 0^+}\lim_{n\to+\infty}B_{n,\delta} = 0.
 \end{equation}
Taking this claim for granted for a moment, we can immediately complete the proof
of the theorem. Indeed, by combining \eqref{eq:limAndelta}\,-\,\eqref{eq:limBndelta},
from \eqref{eq:splitting_the_third} we deduce that
$$\lim_{n\to+\infty}\int_{\Omega}h_n(u_n)f_n \varphi\,dx 
= \int_\Omega h(u)f\varphi\,dx;$$
this, together with \eqref{eq:operatorTerm}\,-\,\eqref{eq:AbsorTerm}, finally shows \eqref{eq:finalGoalWeaku}.

Hence, we are left to show \eqref{eq:limBndelta}. To this end
we first observe that, by using the function $v_n = V_{\delta,\delta}(u_n)\varphi\in\mathcal{X}^{1,2}(\Omega)
\cap L^\infty(\Omega)$ as
a test function in \eqref{eq:Post_limit_in_k} (here, $V_{\delta,\delta}$ is the
function defined in \eqref{eq:def_V_delta_k} with $k = \delta$), we get
\begin{equation}\label{eq:Grad_v_n}
\begin{split}
&
0 \leq \int_{\{u_n \leq \delta\}} h_n(u_n)f_n \varphi\,dx \leq \int_{\Omega}h_n(u_n)f_n V_{\delta,\delta}
(u_n)\varphi\,dx \\
& \qquad
 = \int_{\Omega}\nabla u_n\cdot \nabla v_n\,dx
 + \iint_{\R^{2n}}\dfrac{(u_n(x)-u_n(y))(v_n(x)-v_n(y))}{|x-y|^{n+2s}}\,dx\,dy
 \\
 & \qquad\qquad
 +\int_{\Omega}g(u_n)v_n\,dx \\
 & \qquad (\text{since $V'_{\delta,\delta}(s)\leq 0$ for a.e.\,$s\geq 0$}) \\
 & \qquad
 \leq \int_{\Omega}V_{\delta,\delta}(u_n)\nabla u_n\cdot \nabla\varphi\,dx
 + \iint_{\R^{2n}}\dfrac{(u_n(x)-u_n(y))(v_n(x)-v_n(y))}{|x-y|^{n+2s}}\,dx\,dy
 \\
 & \qquad\qquad
 +\int_{\Omega}g(u_n)v_n\,dx.
 \end{split}
\end{equation}
Moreover, by arguing exactly as in the proof of \cite[Theorem 3.4]{Oliva}
(where the \emph{very same terms} are considered), we have
\begin{equation} \label{eq:convergenzedaOliva}
 \begin{split}
  \mathrm{i)}&\,\,\lim_{\delta\to 0^+}\lim_{n\to+\infty}
  \int_{\Omega}V_{\delta,\delta}(u_n)\nabla u_n\cdot \nabla\varphi\,dx
  = \int_{\{u = 0\}}\nabla u\cdot\nabla\varphi\,dx = 0;
  \\
   \mathrm{ii)}&\,\,\lim_{\delta\to 0^+}\lim_{n\to+\infty}\int_{\Omega}g(u_n)v_n\,dx
   = \int_{\{u = 0\}}g(u)\varphi\,dx = 0.
 \end{split}
\end{equation}
Hence, we are left to study the \emph{nonlocal part} in \eqref{eq:Grad_v_n}, whose presence is
the main different with respect to the case considered in \cite{Oliva}.
\vspace*{0.1cm}

To begin with we observe that, by combining the \emph{global interpolation
inequality} for fractional Sobolev spaces in 
\cite[Theorem 7.45]{LeoniFract}
(applied here with $p_1=p_2=2$ and $\theta = 1-s$) with
the properties of the sequence 
$\{u_n\}_n\subseteq\mathcal{X}^{1,2}(\Omega)$, we get
\begin{equation} \label{eq:convergenceunStrongHs}
\begin{split}
	[u_n-u]_{s,\mathbb{R}^{n}} &\leq C \, \|u_n-u\|_{L^{2}(\mathbb{R}^{n})}^{\theta}\, 
	\|\nabla(u_n-u)\|^{1-\theta}_{L^{2}(\mathbb{R}^{n})} \\
	& \leq C\,\|u_n-u\|_{L^{2}(\mathbb{R}^{n})}^{\theta}\,\rho(u_n-u)^{1-\theta}\\
	& (\text{by property b), and since $\rho(u_n)$ is uniformly bounded}) \\
	& \leq C\,\|u_n-u\|_{L^{2}(\mathbb{R}^{n})}^{\theta}
	\to 0 \quad \textrm{ as } n \to +\infty.
	\end{split}
\end{equation}
Moreover, it is readily seen that also the sequence $\{v_n\}_n$ is bounded in $\mathcal{X}^{1,2}(\Omega)$:
indeed, taking into account \eqref{eq:X12HsRn}, for every $n\in\mathbb{N}$ we have
\begin{equation}
 	\begin{aligned}
 		& \rho(v_n)^2\leq C\,\int_{\Omega}|\nabla v_n|^2\,dx 
 		 \\
 		 &\qquad= \int_{\Omega}\left|V'_{\delta,\delta}(u_n)\varphi \nabla u_{n} \chi_{\{\delta<u_n<2\delta\}} + V_{\delta,\delta}(u_n)\nabla \varphi\right|^2\,dx \\
 		&\qquad \leq C \left( \int_{\Omega}|\nabla u_n|^2\,dx +  \int_{\Omega}|\nabla \varphi|^2\,dx\right)
 		\leq C.
 	\end{aligned}
\end{equation}
As a consequence, by Remark \ref{rem:spaceX12}\,-\,4) (and recalling that
$u_n\to u$ a.e.\,on $\R^n$), we deduce that (by passing to a sub-sequence if needed)
$$\text{$v_n\to v = V_{\delta,\delta}(u)\varphi$ 
strongly in $L^{2}(\Omega)$ (and pointwise a.e.\,in $\Omega$)}. $$
From this,
again by using the \emph{global interpolation
inequality} for fractional Sobolev spa\-ces in 
\cite[Theorem 7.45]{LeoniFract}, we obtain
\begin{equation} \label{eq:convergencevnStrongHs}
\begin{split}
	[v_n-v]_{s,\mathbb{R}^{n}} &\leq C \, \|v_n-v\|_{L^{2}(\mathbb{R}^{n})}^{\theta}\, 
	\|\nabla(v_n-v)\|^{1-\theta}_{L^{2}(\mathbb{R}^{n})} \\
	& \leq C\,\|v_n-v\|_{L^{2}(\mathbb{R}^{n})}^{\theta}\,\rho(v_n-v)^{1-\theta}\\
	& \leq C\,\|v_n-v\|_{L^{2}(\mathbb{R}^{n})}^{\theta}
	\to 0 \quad \textrm{ as } n \to +\infty.
	\end{split}
\end{equation}
Gathering \eqref{eq:convergenceunStrongHs}\,-\,\eqref{eq:convergencevnStrongHs}
(and using H\"older inequality), 
we can then pass to the limit as $n\to+\infty$ in the nonlocal
part in \eqref{eq:Grad_v_n}, obtaining
\begin{align*}
 & \Big|\iint_{\R^{2n}}\dfrac{(u_n(x)-u_n(y))(v_n(x)-v_n(y))}{|x-y|^{n+2s}}\,dx\,dy
 \\
 & \qquad\qquad - \iint_{\R^{2n}}\dfrac{(u(x)-u(y))(v(x)-v(y))}{|x-y|^{n+2s}}\,dx\,dy\Big| \\
 & \qquad\leq 
 \iint_{\R^{2n}}\dfrac{|(u_n-u)(x)-(u_n-u)(y)||v_n(x)-v_n(y)|}{|x-y|^{n+2s}}\,dx\,dy\\
 &\qquad\qquad
 + \iint_{\R^{2n}}\dfrac{|u(x)-u(y)||(v_n-v)(x)-(v_n-v)(y)|}{|x-y|^{n+2s}}\,dx\,dy\\[0.15cm]
 & \qquad
 \leq [u_n-u]_{s,\R^n}[v_n]_{s,\R^n}+[u]_{s,\R^n}[v_n-v]_{s,\R^n}\to 0\quad\text{as $n\to+\infty$}.
\end{align*}
In other words, we have
\begin{equation} \label{eq:limitnNonlocalpart}
\begin{split}
 & \lim_{n\to+\infty}\iint_{\R^{2n}}\dfrac{(u_n(x)-u_n(y))(v_n(x)-v_n(y))}{|x-y|^{n+2s}}\,dx\,dy
 \\
 &\qquad
 = 
 \iint_{\R^{2n}}\dfrac{(u(x)-u(y))(V_{\delta,\delta}(u)(x)\varphi(x)-
  V_{\delta,\delta}(u)(y)\varphi(y))}{|x-y|^{n+2s}}\,dx\,dy.
 \end{split}
\end{equation}
Finally, we turn to pass to the limit as $\delta\to 0^+$ in the right-hand side of 
\eqref{eq:limitnNonlocalpart}. To this aim, we distinguish the following two cases (according
to the values of $s$).
\medskip

-\,\,\textsc{Case I) $0<s<1/2$.} In this first case, we claim that it is possible to
apply the Le\-be\-sgue Dominated Convergence Theorem to show that
\begin{equation} \label{eq:toProveWithDCTLows}
\begin{split}
 & \lim_{\delta\to 0^+}
  \iint_{\R^{2n}}\dfrac{(u(x)-u(y))(V_{\delta,\delta}(u)(x)\varphi(x)-
  V_{\delta,\delta}(u)(y)\varphi(y))}{|x-y|^{n+2s}}\,dx\,dy \\
  & \qquad
  = \iint_{\R^{2n}}\dfrac{(u(x)-u(y))((\chi_{\{u = 0\}}\varphi)(x)-
 (\chi_{\{u = 0\}}\varphi)(y))}{|x-y|^{n+2s}}\,dx\,dy.
\end{split}  
\end{equation}
Indeed, by the very definition of $V_{\delta,\delta}(\cdot)$ we readily see that
$$\text{$\lim_{\delta\to 0^+}V_{\delta,\delta}(u) = \chi_{\{u = 0\}}$ a.e.\,in $\Omega$}.$$
Moreover, since $0\leq V_{\delta,\delta}(\cdot)\leq 1$ and since $\varphi\in L^\infty(\Omega)$, we have
\begin{equation*}
\begin{split}
 & \Big|\dfrac{(u(x)-u(y))(V_{\delta,\delta}(u)(x)\varphi(x)-
  V_{\delta,\delta}(u)(y)\varphi(y))}{|x-y|^{n+2s}}\,dx\,dy\Big| \\
  & \qquad\qquad\leq 2\|\varphi\|_{L^\infty(\Omega)}
  \dfrac{|u(x)-u(y)|}{|x-y|^{n+2s}} = \mathfrak{g}(x,y)
  \end{split}
 \end{equation*}
and \emph{the function $\mathfrak{g}$ is integrable on $\R^n\times\R^n$}. To see this fact it suffices
to notice that,
since $0<s<1/2$, by arguing as in \cite[Proposition 2.2]{DRV} 
(and by recalling the explicit expression of the space $\mathcal{X}^{1,2}(\Omega)$ given in
 \eqref{eq:defX12explicit}) we get
\begin{align*}
 & \iint_{\R^{2n}}\mathfrak{g}(x,y)\,dx\,dy 
 \leq C
 \Big(\int_{\{|x-y|\leq 1\}}\frac{|u(x)-u(y)|}{|x-y|^{n+2s}}\,dx\,dy
 +\|u\|_{L^1(\Omega)}\Big) \\
 & \qquad
 \leq C\Big\{\int_{\R^n}\Big(\int_{B_1(0)}\frac{|u(x)-u(x+z)|}{|z|}\cdot\frac{1}{|z|^{n+(2s-1)}}\,dz\Big)dx
 + \|u\|_{L^1(\Omega)}\Big\} \\
 & \qquad
 \leq C\Big\{\||\nabla u|\|_{L^1(\R^n)}\int_{B_1(0)}\frac{1}{|z|^{n+(2s-1)}}\,dz
 + \|u\|_{L^1(\Omega)}\Big\} < +\infty,
\end{align*}
where $C > 0$ is a suitable constant depending on $\varphi,n$ and $s$.

Gathering all these facts, we are then entitled to apply the 
Le\-be\-sgue Dominated Convergence Theorem in \eqref{eq:limitnNonlocalpart},
which yields \eqref{eq:toProveWithDCTLows}. Summing up, we have
\begin{equation} \label{eq:tousePerConcludereCasoA}
\begin{split}
& \lim_{\delta\to 0^+}
\lim_{n\to+\infty}
 \iint_{\R^{2n}}\dfrac{(u_n(x)-u_n(y))(v_n(x)-v_n(y))}{|x-y|^{n+2s}}\,dx\,dy \\
 & \qquad
 = \iint_{\R^{2n}}\dfrac{(u(x)-u(y))((\chi_{\{u = 0\}}\varphi)(x)-
 (\chi_{\{u = 0\}}\varphi)(y))}{|x-y|^{n+2s}}\,dx\,dy.
 \end{split}
\end{equation}
With \eqref{eq:tousePerConcludereCasoA} at hand, we can finally complete
the proof
of the claimed \eqref{eq:limBndelta} in this case. Indeed, by a direct computation case\,-\,by\,-\,case
we readily derive that
$$\iint_{\R^{2n}}\dfrac{(u(x)-u(y))((\chi_{\{u = 0\}}\varphi)(x)-
 (\chi_{\{u = 0\}}\varphi)(y))}{|x-y|^{n+2s}}\,dx\,dy\leq 0;$$
 thus, by combining this last inequality with \eqref{eq:convergenzedaOliva}, from
 \eqref{eq:Grad_v_n} we obtain
 \begin{align*}
  & 0\leq \lim_{\delta\to 0^+}\lim_{n\to+\infty}
  B_{n,\delta}\\
  & \qquad\qquad\leq \iint_{\R^{2n}}\dfrac{(u(x)-u(y))((\chi_{\{u = 0\}}\varphi)(x)-
 (\chi_{\{u = 0\}}\varphi)(y))}{|x-y|^{n+2s}}\,dx\,dy\leq 0,
 \end{align*}
 from which we immediately derive the claimed \eqref{eq:limBndelta}.
 \medskip
 
 -\,\,\textsc{Case II): $1/2<s<1$}. In this second case, the non-integrability of the 
 singular kernel $|z|^{-n-(2s-1)}$
 near the origin seems to prevent the possibility of applying the
 Lebesgue Dominated Convergence Theorem
 in \eqref{eq:limitnNonlocalpart}; thus, we essentially imitate the approach
 in \cite{Oliva} for the \emph{local part}.
 \vspace*{0.1cm}
 
 First of all, we write
 \begin{align*}
  & \iint_{\R^{2n}}\dfrac{(u(x)-u(y))(V_{\delta,\delta}(u)(x)\varphi(x)-
  V_{\delta,\delta}(u)(y)\varphi(y))}{|x-y|^{n+2s}}\,dx\,dy \\
  & \qquad
  = \iint_{\R^{2n}}V_{\delta,\delta}(u)(x)\dfrac{(u(x)-u(y))(\varphi(x)-
 \varphi(y))}{|x-y|^{n+2s}}\,dx\,dy \\
 & \qquad\qquad
 +
 \iint_{\R^{2n}}\varphi(y)\dfrac{(u(x)-u(y))(V_{\delta,\delta}(u)(x)-
  V_{\delta,\delta}(u)(y))}{|x-y|^{n+2s}}\,dx\,dy \\
  & \qquad (\text{since $V_{\delta,\delta}(\cdot)$ is non-increasing}) \\
  & \qquad\leq \iint_{\R^{2n}}V_{\delta,\delta}(u)(x)\dfrac{(u(x)-u(y))(\varphi(x)-
 \varphi(y))}{|x-y|^{n+2s}}\,dx\,dy.
 \end{align*}
 Moreover, since $0\leq V_{\delta,\delta}(\cdot)\leq 1$, we have
 \begin{align*}
  & \Big|V_{\delta,\delta}(u)(x)\dfrac{(u(x)-u(y))(\varphi(x)-
 \varphi(y))}{|x-y|^{n+2s}}\Big| \\
 & \qquad
  \leq \dfrac{|u(x)-u(y)||\varphi(x)-
 \varphi(y)|}{|x-y|^{n+2s}} = \mathfrak{g}(x,y),
 \end{align*}
 and this function $\mathfrak{g}$ is trivially integrable on $\R^n\times\R^n$ (as both $u$ and $\varphi$
 belong to the space $\mathcal{X}^{1,2}(\Omega)$, see Remark \ref{rem:spaceX12}\,-\,c)).
 As a consequence, we can apply the Lebesgue Dominated Convergence Theorem, obtaining
 \begin{align*}
  & \limsup_{\delta\to 0^+}
  \iint_{\R^{2n}}\dfrac{(u(x)-u(y))(V_{\delta,\delta}(u)(x)\varphi(x)-
  V_{\delta,\delta}(u)(y)\varphi(y))}{|x-y|^{n+2s}}\,dx\,dy \\
  & \qquad \leq \lim_{\delta\to 0^+}\iint_{\R^{2n}}V_{\delta,\delta}(u)(x)\dfrac{(u(x)-u(y))(\varphi(x)-
 \varphi(y))}{|x-y|^{n+2s}}\,dx\,dy \\
 & \qquad = 
 \iint_{\{u = 0\}\times\R^n}\dfrac{(u(x)-u(y))(\varphi(x)-
 \varphi(y))}{|x-y|^{n+2s}}\,dx\,dy.
 \end{align*}
 Summing up, in this second case we have
 \begin{equation*}
\begin{split}
& \limsup_{\delta\to 0^+}
\lim_{n\to+\infty}
 \iint_{\R^{2n}}\dfrac{(u_n(x)-u_n(y))(v_n(x)-v_n(y))}{|x-y|^{n+2s}}\,dx\,dy \\
 & \qquad
 \leq \iint_{\{u = 0\}\times\R^n}\dfrac{(u(x)-u(y))(\varphi(x)-
 \varphi(y))}{|x-y|^{n+2s}}\,dx\,dy.
 \end{split}
\end{equation*}
 With the above estimate at hand, we can  complete
the proof
of the claimed \eqref{eq:limBndelta} also in this case. 
Indeed, since $s > 1/2$, by combining \eqref{eq:assumptionMeasurefzero}\,-\,\eqref{eq:inclusionufzero}, 
we see that
$$|\{u = 0\}| = 0,$$ 
and therefore we clearly have
\begin{equation} \label{eq:NonLocalzeroCaseB}
 \iint_{\{u = 0\}\times\R^n}\dfrac{(u(x)-u(y))(\varphi(x)-
 \varphi(y))}{|x-y|^{n+2s}}\,dx\,dy = 0.
 \end{equation}
 Then, by combining \eqref{eq:NonLocalzeroCaseB} with \eqref{eq:convergenzedaOliva}, from
 \eqref{eq:Grad_v_n} we obtain
 \begin{align*}
  & 0\leq \lim_{\delta\to 0^+}\lim_{n\to+\infty}
  B_{n,\delta}\\
  & \qquad\qquad\leq \iint_{\{u = 0\}\times\R^n}\dfrac{(u(x)-u(y))(\varphi(x)-
 \varphi(y))}{|x-y|^{n+2s}}\,dx\,dy = 0,
 \end{align*}
 from which we immediately derive the claimed \eqref{eq:limBndelta}.
\end{proof}

\subsection{Uniqueness}
Finally, we provide the  
\begin{proof}[Proof of Theorem \ref{thm:Uniqueness}]
Suppose by contradiction that there exist two different solutions of \eqref{eq:Problem} $u_1, u_2$ such that  $g(u_1), g(u_2)\in L^{1}(\Omega)$ and
$$\text{$u_1|_\Omega,u_2	|_\Omega \in H_0^1(\Omega)$
(hence, $u_1,u_2\in \mathcal{X}^{1,2}(\Omega)$}).$$ 
Owing to Lemma \ref{lem:LargerTest}, we can use 
$\varphi = T_{k}(u_1 - u_2)$ as a test function either for the equations solved by $u_1$, either for the one solved by $u_2$. Taking the difference yields
\begin{equation*}
	\begin{aligned}
			&\int_{\Omega}\nabla(u_1-u_2)\cdot \nabla T_{k}(u_1 - u_2)\,dx \\
			& \qquad + \iint_{\mathbb{R}^{2n}}\dfrac{((u_1-u_2)(x)-(u_1-u_2)(y))(T_{k}(u_1-u_2)(x)-T_{k}(u_1-u_2)(y))}{|x-y|^{n+2s}}\,dx\,dy\\
			&\qquad +\int_{\Omega}(g(u_1)-g(u_2))T_{k}(u_1 - u_2)\,dx \\
			&\qquad = \int_{\Omega}(h(u_1)-h(u_2))f T_{k}(u_1 - u_2)\,dx. 
	\end{aligned}
\end{equation*}
Now, since the map $z\mapsto T_k(z)$ is non-decreasing, we have
\begin{equation*}
((u_1-u_2)(x)-(u_1-u_2)(y))(T_{k}(u_1-u_2)(x)-T_{k}(u_1-u_2)(y))\geq 0.
\end{equation*}
Therefore, noticing that
\begin{equation*}
	\int_{\Omega}(g(u_1)-g(u_2))T_{k}(u_1 - u_2)\,dx \geq 0
\end{equation*}
\noindent and 
\begin{equation*}
	\int_{\Omega}(h(u_1)-h(u_2))f T_{k}(u_1 - u_2)\,dx\leq 0,
\end{equation*}
\noindent we find that
\begin{equation*}
	\int_{\Omega}\nabla(u_1-u_2)\cdot \nabla(u_1 - u_2) \chi_{\{|u_1-u_2|\leq k\}}\,dx \leq 0 \quad \textrm{ for every } k \geq 0,
\end{equation*}
\noindent which gives $u_1 =u_2$ a.e. in $\Omega$ after the use of a monotonicity argument. This closes the proof.
\end{proof}

\section{Gain of summability via comparison results}\label{sec:talenti}
Here we establish the pointwise comparison \`a la Talenti mentioned
in the Introduction, allowing to prove Theorem \ref{thm:MoreSummability}.

\begin{theorem}\label{thm:pointwise_talenti}
Let the assumptions of Corollary \ref{cor:ModelProblemEU}, and let $u$ be the unique solu\-tion to the model problem 
\begin{equation}\label{eq:Model-talenti}
\left\{ \begin{array}{rl}
-\Delta u + (-\Delta)^s u + u^q = \dfrac{f(x)}{u^{\gamma}} & \text{ in } \Omega,\\
u>0 & \text{ in } \Omega,\\
u=0 & \text{ on } \mathbb{R}^{n}\setminus \Omega.
\end{array}\right.
\end{equation}
Then
\begin{equation}\label{th:talenti2}
u^\ast(\tau) \leq\left((\gamma+1)v^\ast(\tau)\right)^{\frac{1}{\gamma+1}}, \quad \text{{\color{red}} for a.e. }\tau \in(0,|\Omega|),
\end{equation}
where $v(x)=v^\ast\left(\omega_n|x|^n\right)$ is the unique radial
solution to the problem
\begin{equation}\label{eq:comp-talenti}
\begin{cases}-\Delta v=f^{\sharp} & \text { in } \Omega^{\sharp} \\ v=0 & \text { on } \partial\Omega^{\sharp} .\end{cases}
\end{equation}
\end{theorem}
\begin{proof}
We split the proof into different steps. 
  \vspace{0.2cm}

-\,\,\textsc{Step I): Approximating problems.} For every $k\in \mathbb{N}$ and $x\in\Omega$ we consider the following nonsingular approximating problems 
\begin{equation}\label{pk}
\left\{
\begin{array}{ll}
-\Delta u_k+(-\Delta)^s u_k +g_k(u_k)=\dfrac{f_k}{(u_k+\frac{1}{k})^{\gamma}}\qquad& \mbox{in }\Omega
\\
u_k>0& \mbox{in }\Omega
\\
u_k=0& \mbox{on }\mathbb{R}^n\setminus\Omega,
\end{array}
\right.
\end{equation}
where, recalling \eqref{eq:def_Tk} and exploiting that $f\geq 0$, we have set
\begin{align*}
	\mathrm{i)}\,\,&f_k= T_{k}(f) = \min\{f,k\}; \\
	\mathrm{ii)}\,\,&
	g_{k}(s)= \begin{cases}
		T_{k}(g(s)) & \text{if $s\geq 0$}\\
		0 & \text{if $s < 0$}
	\end{cases} 
	\,\,\,= \begin{cases}
		\min\{s^p,k\} & \text{if $s\geq 0$} \\
		0 & \text{if $s < 0$}
	\end{cases}
\end{align*}
By Lemma \ref{eq:Approx_Problem}, for every $k \in \N$ problem \eqref{pk} has a unique nonnegative solution belonging to $\XX^{1,2}(\Omega)\cap L^{\infty}(\Omega)$. Setting $\tilde u_k:=u_k+\frac{1}{k}$, problem \eqref{pk} reads as
\begin{equation*}
\begin{cases}-\Delta \tilde{u}_{k}+(-\Delta)^s\tilde u_k+g_k\left(\tilde u_k-\frac 1 k\right)=\dfrac{f_k}{\tilde{u}_{k}^{\gamma}} & \text { in } \Omega  \\ \tilde{u}_{k}>\frac{1}{k} & \text { in } \Omega \\ \tilde{u}_{k}=\frac{1}{k} & \text { on } \R^n\setminus \Omega .\end{cases}
\end{equation*}
In particular,
\begin{equation}\label{weakapprox}
\begin{split}
\int_\Omega\nabla \tilde u_k(x)\cdot\nabla\varphi(x)\,dx&+\iint_{\R^{2n}}\frac{(\tilde u_k(x)-\tilde u_k(y))(\varphi(x)-\varphi(y))}{|x-y|^{n+2s}}\,dx\,dy\\&+\int_\Omega g_k\left(\tilde u_k(x)-\frac 1 k\right)\varphi(x) \,dx
=\int_{\Omega}\frac{f_k(x)}{\tilde u_k(x)^{\gamma}}\,\varphi(x)\,dx
\end{split}
\end{equation}
for every $\varphi\in \XX^{1,2}(\Omega)$.

\vspace{0.2cm}

Let us fix $\frac{1}{k}< t<\|u_k\|_{L^\infty(\Omega)}$ and $h>0$. We consider the following test function 
$$\varphi(x)=\mathcal{G}_{t,h}(\tilde u_k(x)),$$ 
where $\mathcal{G}_{t,h}(\theta)$ is defined by
\begin{equation}\label{def-G}
\mathcal{G}_{t,h}(\theta)=\begin{cases}
h \qquad & \mbox{if } \theta>t+h\\
\theta-t & \mbox{if } t<\theta\leq t+h\\
0  & \mbox{if } \theta\leq t.
\end{cases}
\end{equation}
We note that $\mathcal{G}_{t,h}(\tilde u_k)\in \XX^{1,2}(\Omega)$, so we can use it the weak formulation of solution \eqref{weakapprox}, obtaining
\begin{equation}\label{b}
\begin{aligned}
\int_{\Omega}\nabla \tilde u_k(x)&\cdot\nabla\mathcal{G}_{t,h}(\tilde u_k(x))\,dx\\
&+\iint_{\R^{2n}}\frac{(\tilde u_k(x)-\tilde u_k(y))\left(\mathcal{G}_{t,h}(\tilde u_k(x))-\mathcal{G}_{t,h}(\tilde u_k(y))\right)}{|x-y|^{n+2s}}\,dx\,dy\\
&+\int_\Omega
g_k\left(\tilde u_k(x)-\frac 1 k \right)\mathcal{G}_{t,h}(\tilde u_k(x))\,dx\\
&=\int_{\Omega}\frac{f_k(x)}{\tilde u_k(x)^{\gamma}}\,\mathcal{G}_{t,h}(\tilde u_k(x))\,dx.
\end{aligned}
\end{equation}
We first deal with the left-hand side in \eqref{b}. To this end we observe that, setting\begin{equation*}
\mathcal{H}:= (\tilde u_k(x)-\tilde u_k(y))\left(\mathcal{G}_{t,h}(\tilde u_k(x))-\mathcal{G}_{t,h}(\tilde u_k(y))\right),
\end{equation*}
by the monotonicty of $\mathcal{G}_{t,h}$ we see that
$\mathcal{H}\geq 0$.
Therefore, 
\begin{equation*}
 \iint_{\R^{2n}}\frac{(\tilde u_k(x)-\tilde u_k(y))\left(\mathcal{G}_{t,h}(\tilde u_k(x))-\mathcal{G}_{t,h}(\tilde u_k(y))\right)}{|x-y|^{n+2s}}\,dx\,dy\geq 0.
\end{equation*}
Moreover, since $g_k\left(\tilde u_k-\frac{1}{k}\right)\geq 0$ and $\mathcal{G}_{h,t}(\tilde u_k)\geq 0$, we have also
\begin{equation*}
    \int_\Omega g_k\left(\tilde u_k-\frac{1}{k}\right)\mathcal{G}_{t,h}(\tilde u_k(x))\,dx\geq 0,
\end{equation*}
thus
\begin{equation}\label{bb}
\int_\Omega\nabla \tilde u_k(x)\cdot\nabla\mathcal{G}_{t,h}(\tilde u_k(x))\,dx\leq \int_{\Omega}\frac{f_k(x)}{\tilde u_k(x)^{\gamma}}\,\mathcal{G}_{t,h}(\tilde u_k(x))\,dx.
\end{equation}
Dividing by $h>0$ the right-hand side of \eqref{bb} and using \eqref{def-G}, we obtain
\begin{equation*}
\begin{aligned}
\dfrac{1}{h} \int_{\Omega}\dfrac{f_k(x)}{\tilde u_k(x)^{\gamma}}\,\mathcal{G}_{t,h}(\tilde u_k(x))\,dx = &\int_{\{\tilde u_k>t+h\}}\dfrac{f_k(x)}{\tilde u_k(x)^{\gamma}}\,dx\\
&\ \ \ \ +\int_{\{t<\tilde u_k\leq t+h\}}\dfrac{f_k(x)}{\tilde u_k(x)^{\gamma}}\left( \frac{\tilde u(x)-t}{h}\right)\,dx\end{aligned}
\end{equation*}
which implies, 
\begin{equation}\label{lim-RHS}
	\begin{aligned}
    \lim_{h\to 0^+}\dfrac{1}{h}& \int_{\Omega}\dfrac{f_k(x)}{\tilde u_k(x)^{\gamma}}\,\mathcal{G}_{t,h}(u_k(x))\,dx =\int_{\{\tilde u_k>t\}}\dfrac{f_k(x)}{\tilde u_k(x)^{\gamma}}\,dx\\
\end{aligned}
\end{equation}
While for the left-hand side of \eqref{bb}, we have
\begin{equation*}
    \int_\Omega\nabla \tilde u_k(x)\cdot\nabla\mathcal{G}_{t,h}(\tilde u_k(x))\,dx=\int_{\{t<\tilde u_k\leq t+h\}}|\nabla \tilde u_k(x)|^2 \,dx
\end{equation*}
and dividing by $h$ and letting $h\to 0$, we get
\begin{equation}\label{lim-LHS}
    \lim_{h\to0}\frac{1}{h}\int_{\{t<\tilde u_k\leq t+h\}}|\nabla \tilde u_k(x)|^2\, dx=-\frac{d}{dt}\int_{\{\tilde u_k>t\}}|\nabla \tilde u_k(x)|^2\, dx.
\end{equation}
Now by applying the Coarea formula we get
$$
\int_{\{\tilde u_k>t\}}|\nabla \tilde u_k(x)| \,dx=\int_t^{\infty} P(\{x \in \Omega: \tilde u_k(x)>r\}) \,d r, \quad \text{for } t>\frac{1}{k}
$$
where $P(\cdot)$ denotes the perimeter. Then, by the isoperimetric inequality, we have
\begin{equation}\label{use-isop}
n \,\omega_n^{1/n}\, \mu_{\tilde u_k}(t)^{1-1/n} \leq P(\{x \in \Omega: \tilde u_k>t\})=-\frac{d}{d t} \int_{\{\tilde u_k>t\}}|\nabla \tilde u_k(x)| \,dx .
\end{equation}
Moreover, H\"older's inequality gives
\begin{equation}\label{use-holder}
\begin{aligned}
-\dfrac{d}{d t}& \int_{\{\tilde u_k>t\}}|\nabla \tilde u(x)| \,dx  =\lim _{h \rightarrow 0} \dfrac{1}{h} \int_{\{t<\tilde u_k \leq t+h\}}|\nabla \tilde u_k(x)| \,dx \\
& \leq \lim _{h \rightarrow 0}\left(\dfrac{1}{h} \int_{\{t<\tilde u_k \leq t+h\}}|\nabla \tilde u_k(x)|^2 \,dx\right)^{1 /2}\left(\dfrac{\mu_{\tilde u_k}(t)-\mu_{\tilde u_k}(t+h)}{h}\right)^{1/2} \\
& =\left(-\dfrac{d}{d t} \int_{\{\tilde u_k>t\}}|\nabla \tilde u_k(x)|^2 \,d x\right)^{1 / 2} \left(-\mu_{\tilde u_k}^{\prime}(t)\right)^{1 /2}
\end{aligned}
\end{equation}
Thus, combing \eqref{use-isop} and \eqref{use-holder}, we obtain
\begin{equation}\label{deriv-mu}
n^2\omega_n^{2/n} \dfrac{\mu_{\tilde u_k}(t)^{2-2/n}}{-\mu_{\tilde u_k}^{\prime}(t)} \leq -\frac{d}{d t} \int_{\{\tilde u_k>t\}}|\nabla \tilde u_k(x)|^2 \,d x.
\end{equation}
Finally, from \eqref{bb}, \eqref{lim-RHS}, \eqref{lim-LHS} and \eqref{deriv-mu}, we deduce
\begin{equation}\label{ineq-u-k}
    n^2\omega_n^{2/n} \dfrac{\mu_{\tilde u_k}(t)^{2-2/n}}{-\mu_{\tilde u_k}^{\prime}(t)} \leq \int_{\{\tilde  u_k>t\}}\dfrac{f_k(x)}{\tilde u_k(x)^\gamma}\,dx\quad \text{for a.e. }t\in\Big(\frac{1}{k},\|u\|_{L^\infty(\Omega)}\Big)
\end{equation}
Then, recalling Hardy-Littlewood inequality and 
that facts that $f_k \leq f$ and $t>\frac 1k$, we have
$$
n^2 \omega_n^{2 / n} \frac{\mu_{\tilde{u}_k}(t)^{2-2 / n}}{\left(-\mu_{\tilde{u}_k}\right)^{\prime}(t)} 
<\frac{1}{t^\gamma} \int_0^{\mu_{\tilde{u}_k}(t)} f_k^\ast(\sigma) \,d \sigma 
\leq \frac{1}{t^\gamma} \int_0^{\mu_{\tilde{u}_k}(t)} f^\ast(\sigma) \,d \sigma
$$
which can be rewritten as 
\begin{equation}\label{bb2-new}
\tilde{u}_k^\ast(\tau)^\gamma\left(-\tilde{u}_k^\ast(\tau)^{\prime}\right) \leq \frac{1}{n^2 \omega_n^{2 / n}} \,\tau^{-2+2 / n} \int_0^\tau f^\ast(\sigma) \,d\sigma, \quad \text { for a.e. } \tau \in(0,|\Omega|).
\end{equation}
Integrating \eqref{bb2-new} on the interval $(\tau,|\Omega|)$, we get
\begin{equation}\label{stima-tilde-u-k}
\tilde{u}_k^\ast(\tau)^{\gamma+1} \leq \frac{\gamma+1}{n^2 \omega_n^{2 / n}} \int_\tau^{|\Omega|} r^{-2+2 / n}\left(\int_0^r f^\ast(\sigma) \,d \sigma\right) \,d r, \quad\text{for a.e. } \tau\in(0,|\Omega|).
\end{equation}

-\,\,\textsc{Step 2): Local symmetrized nonsingular problems.} Let $v$ be the solution to problem \eqref{eq:comp-talenti}. Due to the radial symmetry of $v$, we have 
$$v(x)=v^\sharp(x)=v^*(\omega_n|x|^n).$$
So, arguing as in \textsc{Step 1)}, we find
\begin{equation*}
n^2 \omega_n^{2 / n} \dfrac{\mu_{v}(t)^{2-2/n}}{-\mu_{v}^{\prime}(t)} =\int_{\{v>t\}} f^\sharp(x)\,d x=\int_0^{\mu_v(t)}f^\ast(\sigma)\, d\sigma
\end{equation*}
which can be rewritten as
$$
-v^\ast(\tau)^{\prime} =\frac{1}{n^2 \omega_n^{2 / n}} \tau^{-2+2 / n} \int_0^\tau f^\ast(\sigma) \,d\sigma, \quad \text {for a.e. } \tau \in(0,|\Omega|)
$$
Again, integrating \eqref{bb2-new} on the interval $(\tau,|\Omega|)$, we get
\begin{equation}\label{eq-v}
v^\ast(\tau)= \frac{1}{n^2 \omega_n^{2 / n}} \int_\tau^{|\Omega|} r^{-2+2 / n}\left(\int_0^r f^\ast(\sigma) \,d \sigma\right) \,d r, \quad\text{for a.e. } \tau\in(0,|\Omega|).
\end{equation}
which together \eqref{stima-tilde-u-k} ensures that 
\begin{equation}\label{stima-punt-u-k}
    \tilde u_k^\ast(\tau)^{\gamma+1}\leq (\gamma+1)\,v^\ast(\tau)\quad \text{ for a.e. }\tau \in(0,|\Omega|).
\end{equation}

-\,\,\textsc{Step 3): Passing to the limit as $k \to +\infty$.} By Lemma \ref{lem:First_Approx}, we know that the sequences $u_k$ is bounded in $\XX^{1,2}(\Omega)$. Hence, up to subsequences, $\tilde u_k$ converges to $\tilde u \in \XX^{1,2}(\Omega)$, weakly in $\XX^{1,2}(\Omega)$, strongly in $L^r(\Omega)$ for any $r\in [1,2^\ast)$ and a.e. in $\Omega$, where $u$, is a weak solution to problems \eqref{eq:Model-talenti}. Thus by property iv) of rearrangement, we have also that, up to subsequence, $\tilde u_k^\ast$ converges strongly to $u^\ast$ in $L^r(0,|\Omega|)$, for any $r\in[1,2^\ast)$, and a.e. in $(0,|\Omega|).$ So, passing to the limit as $k\to\infty$ in \eqref{stima-punt-u-k}, we obtain \eqref{th:talenti2}. This closes the proof.
\end{proof}

\begin{proof}[Proof of Theorem \ref{thm:MoreSummability}]
At first, let us suppose that $1<m<\frac{n}{2}$. Setting 
$$\bar{f}(t)=\frac{1}{t} \int_{0}^{t} f^{*}(s) \,ds,$$ 
by \eqref{th:talenti2}, \eqref{eq-v} and applying Bliss inequality (see \cite{Bli}), we obtain
$$
\begin{aligned}
	\int_{0}^{|\Omega|} u^{*}(s)^{\frac{n m(\gamma+1)}{n-2m}} \, d s & \leq\left(\frac{\gamma+1}{n^{2} \omega_{n}^{2 / n}}\right)^{\frac{n m}{n-2m}} \int_{0}^{|\Omega|}\left(\int_{s}^{|\Omega|} t^{-1+2 / n} \bar{f}(t) \,d t\right)^{\frac{n m}{n-2 m}} \, d s \\
	& \leq\left(\frac{\gamma+1}{n^{2} \omega_{n}^{2 / n}}\right)^{\frac{n m}{n-2 m}} C\,\left(\int_{0}^{|\Omega|} \bar{f}(s)^{p} \, d s\right)^{\frac{n}{n-2 m}} \\
	& =\left(\frac{\gamma+1}{n^{2} \omega_{n}^{2 / n}}\right)^{\frac{n m}{n-2 m}} C\,\|f\|_{L^{m}(\Omega)}^{\frac{n m}{n-2 m}}
\end{aligned}
$$
where the constant $C$ is explicitly given by 
$$ C=
\left( \frac{n(m - 1)}{n - 2m} \right)^{\frac{n(m-1)}{n - 2m}} \cdot
\left( 
\frac{
	\Gamma\left(\frac{n}{2}\right)
}{
	\Gamma\left( \frac{n}{2m} \right)
	\Gamma\left( \frac{n(m - 1)}{2m} + 1 \right)
}
\right)^{\frac{2m}{n - 2m}}.$$ This proves \eqref{coroll-1}.
\medskip

Now let $m>\frac{n}{2}$. By combining the one-dimensional Hardy inequality with H\"older inequality (with exponent $m$ and $m/(m-1)$), we get
$$
\begin{aligned}
	u^{*}(0)^{\gamma+1} & \leq \frac{\gamma+1}{n^{2} \omega_{n}^{2 / n}} \int_{0}^{|\Omega|} t^{-2+2 / n}\left(\int_{0}^{t} f^{*}(r) \,d r\right) \,d t \\
	& \leq \frac{\gamma+1}{n^{2} \omega_{n}^{2 / n}}\Big[\frac{m}{m-1}\left(\frac{n(m-1)}{2 m-n}\right)^{\frac{m-1}{m}}\Big]|\Omega|^{\frac{2 m-n}{n m}}\|f\|_{L^{m}(\Omega)}
\end{aligned}
$$
which is precisely \eqref{coroll-2}.
\medskip

Finally, suppose $m=\frac{n}{2}$. Then, again by H\"older inequality, we have
$$
\begin{aligned}
	u^{*}(s)^{\gamma+1} & \leq \frac{\gamma+1}{n^{2} \omega_{n}^{2 / n}}\left(\int_{s}^{|\Omega|} \bar{f}(t)^{n / 2} \, d t\right)^{2 / n}\left(\int_{s}^{|\Omega|} \frac{1}{t} \, d t\right)^{\frac{n-2}{n}} \\
	& \leq \frac{\gamma+1}{n^{2} \omega_{n}^{2 / n}}\left(\log \frac{|\Omega|}{s}\right)^{\frac{n-2}{n}}\|f\|_{L^{n / 2}(\Omega)}, \quad \text{for a.e. }s \in(0,|\Omega|)
\end{aligned}
$$
or, equivalently,
$$
\frac{u^{*}(s)}{\left(\log \frac{|\Omega|}{s}\right)^{\frac{n-2}{n(\gamma+1)}}} \leq\left(\frac{\gamma+1}{n^{2} \omega_{n}^{2 / n}}\right)^{\frac{1}{\gamma+1}}\|f\|_{L^{n / 2}(\Omega)}^{\frac{1}{\gamma+1}}, \quad \text{for a.e. }s \in(0,|\Omega|) .
$$
By \cite[Proposition 2.2]{BCT}, we get $u \in L_{\exp t^{\frac{n(\gamma+1)}{n-2}}}(\Omega)$ and \eqref{coroll-3} holds true. 
\end{proof}

\appendix
\section{Concerning Assumption \eqref{eq:assumptionMeasurefzero}}\label{appendix}
As is clear from the proof of Theorem \ref{thm:Existence} (where it 
is used in place of the Dominated Convergence Theorem to handle 
the nonlocal term \eqref{eq:limitnNonlocalpart}), assumption 
\eqref{eq:assumptionMeasurefzero} is purely technical and plays 
no essential role. Although it is general enough to include the `model' singularity
$$h(u) = \frac{1}{u^\gamma},$$
we now briefly outline, for the sake of completeness, how to prove the existence of a (unique) 
solution in the more general `model' case
$$\frac{f(x)}{u^\gamma},$$
where $f$ may vanish on a set of positive measure. In this broader setting, however, we must consider the 
`model' absorption term
$$g(s) = s^q.$$
Let then $0\leq \gamma\leq 1$ be fixed, and let $f\in L^m(\Omega)\,(\text{with $m > 1$}),\,\text{$f\geq 0$ a.e.\,in $\Omega$}$.
Moreover, according to assumption (H)$_g$
we
let $q\geq 0$
be such that
\begin{equation} \label{eq:assqgeqh2}
	q\geq \frac{1-m\gamma}{m-1}.
\end{equation}
Finally, \emph{only in this part of the appendix} we also suppose that
\begin{align} 
	\ast)&\,\,q\leq 2^\ast-1 \label{eq:qsubcriticalModello} \\
	\ast)&\,\,\text{there exists $\delta \in (0,1)$ s.t.\,$f\in C^{\delta}_{\mathrm{loc}}(\Omega)$}
\end{align}
(notice that the validity of \emph{both} \eqref{eq:assqgeqh2} and \eqref{eq:qsubcriticalModello} 
`encodes'
an implicit assumptions on $m$ and $\gamma$).
Under such assumptions, our aim is to demonstrate that
there exists a solution
(in the sense of Definition \ref{def:Distributional_Solution}) of 
the `model' problem \eqref{eq:ModelProblem}.
\vspace*{0.05cm}

To this end, for every $k,n\in\mathbb{N}$, we start considering the truncated problem
\begin{equation}\label{eq:Approx_Problem_Specifico}
	\begin{cases}
		\mathcal{L}u + g_{k}(u) = \dfrac{f_n}{(u+1/n)^\gamma} & \textrm{in } \Omega,\\
		u \geq 0 & \textrm{in } \Omega,\\
		u = 0 & \textrm{in } \mathbb{R}^{n}\setminus \Omega,
	\end{cases}
\end{equation}  
where (as in the general case) we have set
\begin{align*}
	\mathrm{i)}\,\,&f_n= T_{n}(f); \\
	\mathrm{ii)}\,\,&
	g_{k}(s)= \begin{cases}
		T_{k}(g(s)) & \text{if $s\geq 0$}\\
		0 & \text{if $s < 0$}
	\end{cases} 
	\,\,\,= \begin{cases}
		\min\{s^p,k\} & \text{if $s\geq 0$} \\
		0 & \text{if $s < 0$}.
	\end{cases}
\end{align*}
Since the `regularized' function $\mathcal{R}_n(t) = (t+1/n)^\gamma$ (with $t\geq 0$)
clearly satisfies the same `qua\-li\-tative' pr\-o\-per\-ties of the truncation $h_n$ defined in 
\eqref{eq:fngnhnGeneral}, namely
\begin{equation}\label{eq:propRnPerLemma}
	\text{$\mathcal{R}_n\geq 0$ and $\mathcal{R}_n\in L^\infty(\Omega)$},
\end{equation}
we can certainly exploit Lemmas \ref{lem:First_Approx}\,-\,\ref{lem:propun} 
also in this context (notice that the proof of such lemmas only relies on
\eqref{eq:propRnPerLemma},
and not on the explicit definition of $h_n$). Thus, for every fixed $n\in\mathbb{N}$,
we can
find a function $u_n\in\mathcal{X}^{1,2}(\Omega)$, obtained
as the limit as $k\to+\infty$ of $u_{k,n}$ (the latter being a solution
of
\eqref{eq:Approx_Problem_Specifico}),
such that
\begin{itemize}
	\item[(P1)] $u_n\in L^\infty(\Omega)$;
	\item[(P2)] $u_n$ is a weak solution (in the sense of Definition \ref{def:Distributional_Solution}) of
	\begin{equation*}
		\begin{cases}
			\mathcal{L}u+ u^q  = \dfrac{f_n}{(u+1/n)^\gamma} & \textrm{in } \Omega,\\
			u \geq 0 & \textrm{in } \Omega,\\
			u = 0 & \textrm{in } \mathbb{R}^{n}\setminus \Omega,
		\end{cases}
	\end{equation*}  
	that is, $u_n\geq 0$ a.e.\,in $\Omega$ and
	\begin{equation} \label{eq:WeakFormPbApproxnModello}
		\mathcal{B}(u_n,\varphi)+\int_\Omega u_n^q\varphi\,dx =\int_\Omega \dfrac{f_n}{(u_n+1/n)^\gamma}\varphi\,dx\quad 
		\text{for every $\varphi\in \mathcal{X}^{1,2}(\Omega)$}.
	\end{equation}
	\item[(P3)] there exists a constant $C>0$, independent of $n$, such that
	\begin{align}
		&\rho(u_n)^2 + \|u_n^{q+1}\|_{L^{1}(\Omega)} \leq C;  \label{eq:Apriori_estimatesModello} \\
		&\int_{\Omega}\dfrac{f_n}{(u_n+1/n)^\gamma}\varphi\,dx
		\leq C + \rho(\varphi)^2 + C\|\varphi\|_{L^{\infty}(\Omega)}.
		\label{eq:Apriori_estimates_2Modello}
	\end{align}
\end{itemize}
In view of these facts, we can then proceed exactly as in the \emph{incipit} of the proof of Theorem
\ref{thm:Existence}, and find a function 
$u\in\mathcal{X}^{1,2}(\Omega)$ such that (up to a subsequence)
\begin{itemize}
	\item[a)] $u_n \to u$ weakly in $\mathcal{X}^{1,2}(\Omega)$;
	\item[b)] $u_n \to u$  strongly in $L^{r}(\Omega)$ for every $1\leq r<2^{\ast}$;
	\item[c)] $u_n\to u$ pointwise a.e.\,in $\Omega$ (hence, in $\R^n$);
	\item[d)] $g(u) = u^q\in L^1(\Omega)$.
\end{itemize} 
In order to pass to the limit as $n\to+\infty$ in 
\eqref{eq:WeakFormPbApproxnModello} (and thus to prove
that $u$ solves problem \eqref{eq:ModelProblem}), instead, we rely on the following lemma.
\begin{lemma} \label{lem:propunModelCase}
	The following assertions hold.
	\begin{enumerate}
		\item the sequence $\{u_n\}_n$ is non-decreasing in $\Omega$;
		\item for every compact set $K\subseteq\Omega$ there exists $c = c_K > 0$ such that
		\begin{equation} \label{eq:staccadazero}
			\text{$u_n\geq c_K$ a.e.\,in $K$}.
		\end{equation}
	\end{enumerate}
\end{lemma}
\begin{proof}
	We prove the two assertions separately.
	\medskip
	
	(1)\,\,Let $n\in\mathbb{N}$ be fixed, and let $v_n = u_{n+1}-u_n$.
	First of all,
	by using 
	$$\varphi = v_n^- = (u_{n+1}-u_n)^-\in\mathcal{X}^{1,2}(\Omega)$$
	as a test func\-tion in the weak formulation \eqref{eq:WeakFormPbApproxnModello}
	for $u_n$ and $u_{n+1}$, respectively,
	and by subtracting the corresponding identities, we derive 
	\begin{equation} \label{eq:doveusareisegniMonotone}
		\begin{split}
			& \int_{\Omega}|\nabla(u_{n+1}-u_n)^-|^2\,dx
			- \iint_{\R^{2n}}\frac{(v_n(x)-v_n(y))(v_n^-(x)-v_n^-(y))}{|x-y|^{n+2s}}\,dx\,dy
			\\
			&\quad = \mathcal{B}(u_n,\varphi)-\mathcal{B}(u_{n+1},\varphi)
			\\
			& \quad
			= \int_{\{u_n\geq u_{n+1}\}}\big(u_n^q-u_{n+1}^q\big)v_n\,dx\\
			&\quad\quad -\int_{\{u_n\geq u_{n+1}\}}f_n\Big(\frac{1}{(u_n+1/n)^\gamma}
			- \frac{1}{(u_{n+1}+1/(n+1))^\gamma}\Big)v_n\,dx.
		\end{split}
	\end{equation}
	We then observe that, by the computations in Lemma \ref{lem:First_Approx}, we get
	\begin{equation} \label{eq:nonlocalsignMonotone}
		\begin{split}
			& -\big(v_n(x)-v_n(y)\big)\big(v_n^-(x)-v_n^-(y)\big)
			\geq (v_n(x)-v_n(y))^2\geq 0.
		\end{split}
	\end{equation}
	Moreover, since the map $s\mapsto g(s) = s^q$ is increasing in $(0,+\infty)$, we have
	\begin{equation} \label{eq:signuqMonotone}
		\begin{split}
			\big(u_n^q-u_{n+1}^q\big)v_n = (g(u_n)-g(u_{n+1}))(u_{n+1}-u_{n})\leq 0.
		\end{split}
	\end{equation}
	Finally, since the map $s\mapsto\mathcal{R}_{n}(s) = (s+1/n)^{-\gamma}$ is decreasing
	in $(0,+\infty)$, and since (by definition) $f_n = T_n(f)\geq 0$ in $\Omega$,
	for a.e.\,$x\in\{u_n\geq u_{n+1}\}$ we also have
	\begin{equation} \label{eq:singSingularMonotone}
		\begin{split}
			& -f_n\Big(\frac{1}{(u_n+1/n)^\gamma}
			- \frac{1}{(u_{n+1}+1/(n+1))^\gamma}\Big)v_n \\[0.2cm]
			& \qquad (\text{since $v_n = u_{n+1}-u_{n}\leq 0$}) \\
			& \qquad \geq f_n\Big(\frac{1}{(u_n+1/n)^\gamma}
			- \frac{1}{(u_{n+1}+1/n)^\gamma}\Big)(u_n-u_{n+1}) \\
			& \qquad = f_n\big(\mathcal{R}_n(u_n)-\mathcal{R}_n(u_{n+1})\big)(u_n-u_{n+1})\leq 0.
		\end{split}
	\end{equation}
	Gathering \eqref{eq:nonlocalsignMonotone}-\eqref{eq:singSingularMonotone}, from
	\eqref{eq:doveusareisegniMonotone} we thus obtain
	\begin{align*}
		& 0\leq \int_{\Omega}|\nabla(u_{n+1}-u_n)^-|^2\,dx \\
		& \qquad
		\leq \int_{\Omega}|\nabla(u_{n+1}-u_n)^-|^2\,dx
		- \iint_{\R^{2n}}\frac{(v_n(x)-v_n(y))(v_n^-(x)-v_n^-(y))}{|x-y|^{n+2s}}\,dx\,dy\leq 0.
	\end{align*}
	As a consequence, we conclude that
	$$\int_{\Omega}|\nabla(u_{n+1}-u_n)^-|^2\,dx  = 0,$$
	and this shows that $u_{n+1}\geq u_n$ a.e.\,in $\Omega$, as desired.
	\medskip
	
	(2)\,\,First of all we observe that, since \emph{we have already proved}
	that$\{u_n\}_n$ is non-de\-cre\-asing in $\Omega$,
	it suffices to prove \eqref{eq:staccadazero} for $u_1$.
	\vspace*{0.05cm}
	
	To this end, we claim that there exists
	$\alpha\in (0,1)$ such that
	\begin{equation} \label{eq:regulu1SMP}
		u_1\in C^{1,\alpha}(\overline{\Omega})\cap C^{2,\alpha}_{\mathrm{loc}}(\Omega).
	\end{equation}
	Indeed, by property (P2) we have that $u_1\in\mathcal{X}^{1,2}(\Omega)$ is a solution of
	$$\begin{cases}
		\mathcal{L} u = \mathcal{F}(x,u) & \text{in $\Omega$}, \\
		u = 0 & \text{in $\R^n\setminus\Omega$}
	\end{cases}\qquad\text{where $\mathcal{F}(x,t) = \frac{f_1(x)}{(|t|+1)^\gamma}-|t|^q$}.$$
	Thus, since
	$\mathcal{F}$ satisfies the following growth estimate (by the assumption on $q$)
	\begin{align*}
		|\mathcal{F}(x,t)|&\leq |f_1(x)|+t^q = T_1(f)+|t|^q\\
		& (\text{since $T_1(f) = \min\{f,1\}$}) \\
		& \leq 1+|t|^q\leq 2(1+|t|^{2^\ast-1}),
	\end{align*}
	we derive from \cite[Theorem 1.3]{SVWZ2} 
	that  
	\begin{equation} \label{eq:u1C1alfa}
		\text{$u_1\in C^{1,\alpha}(\overline{\Omega})$ for a suitable $\alpha\in (0,1)$}
	\end{equation}
	(notice that here we are also using the fact that $\de\Omega$ is smooth).
	
	On the other hand, since $f\in C^{\delta}_{\mathrm{loc}}(\Omega)$,
	it is easily seen that $\mathcal{F}$ is locally H\"older-con\-ti\-nuous on $\Omega\times\R$
	(for a suitable exponent depending on $\gamma$ and $\delta$). Thus,
	since $u_1\in L^\infty(\Omega)$ (see property I)), we can apply also 
	\cite[Theorem 1.5]{SVWZ2}, obtaining 
	\begin{equation} \label{eq:u1C2alfaloc}
		u\in C^{2,\alpha}_{\mathrm{loc}}(\Omega)
	\end{equation}
	(up to shrinking $\alpha$ if needed). Gathering \eqref{eq:u1C1alfa}-\eqref{eq:u1C2alfaloc}, we then
	get \eqref{eq:regulu1SMP}.
	\vspace*{0.1cm}
	
	Now we have proved \eqref{eq:regulu1SMP}, we can easily conclude
	the demonstration of \eqref{eq:staccadazero}. Indeed, using \eqref{eq:regulu1SMP},
	we see that $u_1$ is a \emph{classical solution} of 
	\begin{equation} \label{eq:PbEsplicitou1}
		\begin{cases}
			\mathcal{L}u+ u^q  = \dfrac{f_1}{(u+1)^\gamma} & \textrm{in } \Omega,\\
			u \geq 0 & \textrm{in } \Omega,\\
			u = 0 & \textrm{in } \mathbb{R}^{n}\setminus \Omega,
		\end{cases}
	\end{equation}
	In particular, we have $u_1 = 0$ pointwise in $\R^n\setminus\Omega$, and $u_1\gneqq 0$
	pointwise in $\Omega$ (notice that we necessarily have $u_1\not\equiv 0$, since
	the zero function \emph{does not solve} \eqref{eq:PbEsplicitou1}).
	
	We then argue by contradiction, supposing that there exists $x_0\in\Omega$ such that $u_1(x_0) = 0$
	(so that $x_0$ is an minimum point for $u_1$ in $\Omega$). Thus, we have
	\begin{align*}
		& (-\Delta)^s u(x_0) \geq  -\Delta u(x_0) + (-\Delta)^s u(x_0) \\
		& \qquad = -u(x_0)^q+\frac{f_1(x_0)}{(u_1(x_0)+1)^\gamma} 
		\\
		& \qquad = \frac{f_1(x_0)}{(u_1(x_0)+1)^\gamma} \geq 0;
	\end{align*}
	on the other hand, we also have
	$$0\leq 2\,\mathrm{P.V.}\,\int_{\R^n}\frac{u(x_0)-u(y)}{|x_0-y|^{n+2s}}\,dy\leq 0,$$
	and hence $u(y) = u(x_0) = 0$ for every $x\in\Omega$, which is a contradiction.
	
	Summing up, we have proved that $u_1 > 0$ pointwise in $\Omega$, and
	this immediately implies the desired \eqref{eq:staccadazero}
	(in view of the regularity of $u_1$ in $\Omega$).
\end{proof}
As anticipated, by exploiting Lemma \ref{lem:propunModelCase} we can easily prove
that the function $u\in\mathcal{X}^{1,2}(\Omega)$, obtained
as the limit as $n\to+\infty$ of the sequence $\{u_n\}_n$ (which, in this case,
exists without the need
of passing to a sub-sequence, see Lemma \ref{lem:propun}-(1))
is a weak solution of the `model' problem \eqref{eq:ModelProblem}.
\vspace*{0.1cm}

Indeed, we already know from a) and d) that $u\in\mathcal{X}^{1,2}(\Omega)$ and
$u^q\in L^1(\Omega)$. Moreover, given any 
$\varphi\in C_0^\infty(\Omega)$ such that
$\varphi\geq 0$ in $\Omega$, the weak convergence
of $u_n$ to $u$ in $\mathcal{X}^{1,2}(\Omega)$ gives (as in the
proof of Theorem \ref{thm:Existence})
$$\lim_{n\to+\infty}\mathcal{B}(u_n,\varphi) = \mathcal{B}(u,\varphi).$$
We now observe that, as regards the absorption term, a direct application
of the Fatou Lemma  based on Lemma \ref{lem:propun}-(1) gives
$$\lim_{n\to+\infty}\int_\Omega u_n^q\varphi\,dx = \int_\Omega u^q\varphi\,dx.$$
Furthermore, by Lemma \ref{lem:propun}-(2) we get
$$\frac{f_n}{(u_n+1/n)^\gamma}\varphi\leq \frac{f_n}{u_n^\gamma}\varphi\leq 
c_{\mathrm{supp}(\varphi)}^{-\gamma}f\varphi;$$
thus, we can apply the Dominated Convergence  Theorem, obtaining
$$\lim_{n\to+\infty}\int_{\Omega}\frac{f_n}{(u_n+1/n)^\gamma}\,dx
= \int_\Omega\frac{f}{u^\gamma}\,dx$$
(where we have also used the fact that $f_n = T_n(f)\to f$ a.e.\,in $\Omega$).
\vspace*{0.1cm}

Gathering all these facts, we can then pass to the limit as $n\to+\infty$
in \eqref{eq:WeakFormPbApproxnModello}, 
thus concluding that $u$ is a weak solution of problem \eqref{eq:ModelProblem},
as desired.
\vspace*{0.1cm}

We explicitly notice that, since $\{u_n\}_n$ is non-decreasing, by Lemma \ref{lem:propun}-(2)
(ap\-pl\-ied here with $n = 1$) we derive that $u\geq u_1 > 0$ a.e.\,in $\Omega$.
\begin{remark} \label{rem:GeneralAppendix}
	It is worth highlighting that the above computations can be carried out also for the more general problem \eqref{eq:Problem}, provided that
	$g$ and $h$ fulfill, together with (H)$_h$ and (H)$_g$, the following \emph{additional assumptions}:
	\begin{enumerate}
		\item $g$ is \emph{non-decreasing} and locally H\"older-continuous on $[0,+\infty)$; moreover, there exists a positive constant $C > 0$ such that	
$$0\leq g(s)\leq C(1+s^{2^*-1})\quad \text{for all $s\geq 0$}.$$
\item $h$ is \emph{non-increasing} and locally H\"older-continuous on $(0,+\infty)$.
\end{enumerate}
	Under these assumptions, one can then prove
	the existence of a solution of problem \eqref{eq:Problem} without requiring \eqref{eq:assumptionMeasurefzero} on the function $f$.
	\end{remark}

\end{document}